\documentclass{amsart}
\pdfoutput=1
\usepackage{latexsym}
\usepackage{amsmath,amsthm,amsfonts,amscd,eucal}
\usepackage{graphicx}
\usepackage{epsfig}

\numberwithin{equation}{section}

\def\ca{{\mathcal A}}
\def\cb{{\mathcal B}}
\def\cc{{\mathcal C}}
\def\cd{{\mathcal D}}

\def\cf{{\mathcal F}}

\def\cai{{\mathcal I}}

\def\cn{{\mathcal N}}

\def\car{{\mathcal R}}
\def\cs{{\mathcal S}}

\def\bc{{\mathbb C}}

\def\bn{{\mathbb N}}

\def\br{{\mathbb R}}
\def\bt{{\mathbb T}}

\def\bz{{\mathbb Z}}

\def\ga{{\mathfrak A}}

 \def\gph{{\mathfrak h}}

\def\a{\alpha}
\def\b{\beta}
\def\g{\gamma}        \def\G{\Gamma}
\def\d{\delta}        \def\D{\Delta}
\def\eps{\varepsilon}

\def\th{\vartheta}    
\def\j{\iota}

\def\l{\lambda}       \def\La{\Lambda}
\def\m{\mu}
\def\n{\nu}

\def\r{\rho}
\def\s{\sigma}        
     
\def\t{\tau}

\def\f{\varphi}

\def\om{\omega}        

\newtheorem{Thm}{Theorem}[section]
\newtheorem{Cor}[Thm]{Corollary}
\newtheorem{Prop}[Thm]{Proposition}
\newtheorem{Lemma}[Thm]{Lemma}

\theoremstyle{definition}
\newtheorem{Dfn}[Thm]{Definition}
\newtheorem{Exmp}[Thm]{Example}

\theoremstyle{remark}
\newtheorem{Rem}[Thm]{Remark}

\newcommand{\set}[1]{\left\{#1\right\}}
\newcommand{\norm}[1]{\| #1 \|}
\def\itm#1{\item{$(#1)$}}

\newcommand{\conv}{\text{conv}\,}
\newcommand{\ssv}{\text{span}\,}
\newcommand{\diam}{\text{diam}\,}

\newcommand{\meno}{\setminus}
\newcommand{\vuoto}{\emptyset}

\newcommand{\ccr}{\textrm{CCR}}
\newcommand{\comb}{\dashv}

\newcommand{\coimplies}{\ \Longleftarrow\ }

\def\supp{\mathop{\rm supp}}

\newcommand{\nn}{\nonumber}

\def\im{\mathop{\rm Im}}
\DeclareMathOperator{\vol}{vol}

\begin{document}

\title[Bose Einstein condensation]
{Bose Einstein condensation on inhomogeneous amenable graphs}
\author{Francesco Fidaleo, Daniele Guido, Tommaso Isola}
\address{Dipartimento di Matematica,
Universit\`{a} di Roma Tor Vergata, 
Via della Ricerca Scientifica 1, Roma 00133, Italy} 

\email{fidaleo@mat.uniroma2.it, guido@mat.uniroma2.it,
isola@mat.uniroma2.it} 

\dedicatory{Dedicated to the memory of our friend and colleague Claudio D'Antoni (1950-2010)}

\thanks{The authors were partially supported by MIUR, GNAMPA, by
the European Network ``Noncommutative Geometry" MRTN--CT--2006-031962, and by 
 the ERC Advanced Grant 227458 OACFT ``Operator Algebras and
Conformal Field Theory''.}%


\date{\today}

\begin{abstract}
We investigate the Bose--Einstein Condensation on 
nonhomogeneous amenable networks for the model describing arrays of 
Josephson junctions. The resulting topological model, whose 
Hamiltonian is the pure hopping one given by the opposite of the 
adjacency operator, has also a mathematical interest in itself. We 
show that for the nonhomogeneous networks like the 
comb graphs, particles condensate in momentum and configuration 
space as well. 
In this case different properties of the network, of geometric and probabilistic nature, such as the volume growth, the shape of the ground state, and the transience, all play a r\^ole in the condensation phenomena.
The situation is quite 
different for homogeneous networks where just one of these parameters, e.g. the volume growth, is enough to determine the appearance of the condensation.
\end{abstract}

\maketitle

\section{Introduction}
This paper is devoted to the analysis of thermodynamical states on complex networks with pure hopping Hamiltonian, in particular of those exhibiting Bose--Einstein condensation (BEC for short). Here the network is described by an infinite topological graph $X$, 
where we consider free Bosons  described by the Canonical Commutation Relations on (a suitable dense subspace of) $\ell^2(VX)$, see \cite{BR1, BR2} for references. 
The so called {\it pure hopping Hamiltonian}  is the free Hamiltonian described, on the one particle space $\ell^2(VX)$, by
\begin{equation}
\label{phop}
H:=\|A\|I-A\,,
\end{equation}
where $A$ is the adjacency operator acting on $\ell^2(VG)$, and the normalization constant is chosen in order to get a positive Hamiltonian $H$.

The study of BEC on infinite graphs with pure hopping Hamiltonian, in particular for the so called comb graph (see fig.~\ref{Fig8}),
   \begin{figure}[ht]
 	 \centering
	 \psfig{file=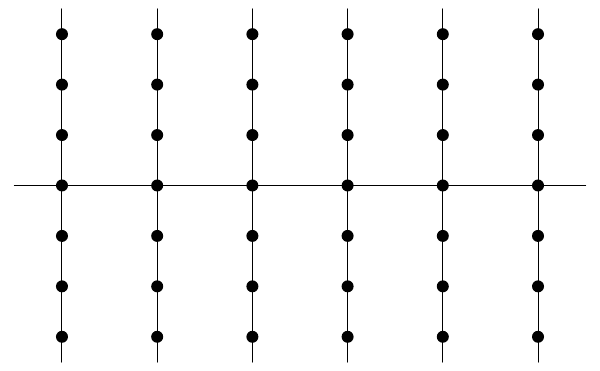,height=1.5in} 
	 \caption{Comb graph.}
	 \label{Fig8}
     \end{figure}
started with a series of papers (cf. \cite{BCRSV} and references therein) motivated by  the relevance of the comb graphs in describing physical phenomena such as arrays of Josephson junctions, and then continued in \cite{ABO,M}. As observed in \cite{BCRSV}, the simultaneous choice of an inhomogeneous graph (the comb graph) and of the pure hopping Hamiltonian can produce the finiteness of the critical density, even at low dimensions.

We now recall that, on a $\bz^d$-periodic network, the following features are equivalent:
\begin{itemize}
\item the finiteness of the critical density;
\item the transience of the graph in the setting of the theory of the random 
walks;\footnote{More precisely, it is equivalent to the transience of random walk on the 
$\bz^d$--periodic graph under consideration whose generator is the Laplacian.} 
\item the dimension being greater than 2; 
\item {\it Bose--Einstein Condensation}, that is the existence of locally normal states (i.e. states with finite local density of particles) for which the percentage of particles occupying the ground state in the infinite volume limit is strictly positive.  In particular, such states may exhibit any particle density greater than the critical one.
\end{itemize}
Let us notice that both finite critical density and transience are related to the behavior of the Hamiltonian in a neighborhood of the bottom of the spectrum, cf. Definition \ref{criticaldensity} below, namely the BEC phenomena are deeply connected with the  spectrum of the Hamiltonian near zero. 

When considering the pure hopping Hamiltonian in \eqref{phop}, this  corresponds  to the spectral properties of the adjacency $A$, for values close to $\|A\|$.

The present paper is devoted to the investigation of the BEC, and its related features, for models whose underlying graphs are ``small'' additive perturbations of periodic graphs, which we shall call  {\it essentially periodic} graphs (cf. Definition \ref{essentiallyperiodic}), and where the Hamiltonian is the pure hopping one, based on the adjacency operator according to \eqref{phop}.

We prove that, in the infinite volume limit, the density of the eigenvalues (called 
{\it integrated density of the states}  in the physical literature) of the adjacency matrix $A_p$ of the perturbed graph is the same as that of the unperturbed graph $A$, up to a possible shift of the spectrum. This is due to the possible difference between the two norms $\|A_p\|$, $\|A\|$ which enter in the definition of the Hamiltonian \eqref{phop}. Such a possible translation of the spectrum is important because corresponds to  the appearance of the so called {\it hidden spectrum} according to the physical literature. The presence of the hidden spectrum always implies the finiteness of the critical density (see \eqref{dc}). 

The main result of the present work is the general analysis of the features related to BEC for essentially periodic networks with pure hopping Hamiltonian. First, we give the general proof of the following fact. Namely, for essentially periodic graphs, the increase of the norm $\|A_p\|$ of the perturbed adjacency is equivalent to the appearance of the hidden spectrum, which in turn implies the finiteness of the critical density. We also show the presence of hidden spectrum in many examples, including the comb graph, where such property was already noticed in \cite{BCRSV}. This implies finite critical density also for dimension $d\leq2$. Second, we show that, because of the essential difference between the adjacency and the Laplace operator due to the inhomogeneity introduced by the perturbation, all the properties described above are not equivalent, hence have to be separately analyzed and proved.

As mentioned above, in the case of periodic graphs the appearance of BEC is determined by the finiteness of the critical density. If this happens, in order to obtain a thermodynamical state with non-trivial condensate, one should choose the chemical potential $\mu_{\Lambda_n}$ for the cut-off region $\Lambda_n$ so to keep the density constant and above the critical density.

For inhomogeneous graphs with pure hopping Hamiltonian the finiteness of the critical density is neither necessary nor sufficient  for the existence of a thermodynamic limit with BEC.
Comb graphs for example have finite critical density also at low dimensions, however this does not imply the existence of locally normal thermodynamical states exhibiting BEC. This is because such finiteness is due to the presence of hidden spectrum, instead of to the integrability of the divergence at the bottom of the spectrum.

On the other hand, one can prove (cf. \cite{FGI}) that for the graph $\bn$ a thermodynamical state with BEC exists, despite of the infinite critical density.

Moreover, higher-dimensional combs admit BEC, but the threshold dimension is not given by the growth of the volume, but by the growth of the Perron-Frobenius vector. Such threshold dimension also plays a r\^ole in the choice of the sequence $\{\mu_{\Lambda_n}\}$ of the finite volume chemical potentials which gives rise to locally normal states with BEC.

Apparently, the notion which is still capable of determining the existence of BEC is that of transience of the Hamiltonian operator.

The class of graphs that we analyze in the paper consists of zero density additive perturbations of 
$\Gamma$--periodic graphs with finite quotients, $\Gamma$ being a discrete amenable group, typically $\bz^d$. We equip any such graph $X=(V,E)$ with a C$^*$-algebra of operators on $\ell^2(VX)$, containing in particular $\Gamma$--invariant operators with finite propagation, and endow the algebra with a finite trace $\tau$.
This trace, composed with the spectral projections of the Hamiltonian, produces the spectral measure at the infinite volume limit. Therefore the presence of hidden spectrum may be seen as a consequence of the non faithfulness of the GNS representation associated with the trace $\tau$.

A relevant notion in this paper is that of generalized Perron--Frobenius eigenvector, namely a (not necessarily $\ell^2$) vector with positive entries on $VX$ which is an eigenvector of the adjacency operator $A$, with eigenvalue $\|A\|$. For the pure hopping Hamiltonian indeed, this vector describes the minimal energy, hence the Bose--Einstein condensate.

Another main technical issue studied in this paper is how small perturbations of periodic graphs modify the behavior under the thermodynamical limit. As a main tool, we prove what we call the secular equation, which relates the resolvent of the adjacency operator of a perturbed graph with that of the unperturbed one. 

As a first byproduct we prove the existence of hidden spectrum, and hence of finite critical density, for a large class of examples.
Finally, we can completely analyze the thermodynamical limit for the comb graph $\bz^d
 \comb \bz$, showing that locally normal KMS states at the critical density may appear exactly when $d\geq3$.

In particular, we show that for  the low dimensional combs, it is impossible to exhibit infinite volume states describing a portion of Bose--Einstein condensate having a correct local meaning.  This happens since the adjacency matrix is  recurrent, and we can exhibit a unique KMS state  for a given inverse temperature $\b$, which is non normal w.r.t. the Fock state.  In addition, the thermodynamic limit with constant density does not describe Bose--Einstein condensate even in the transient case, i.e. for the comb graphs $\bz^d  \comb \bz$, $d\geq3$.  
Such a condensate exists only at the critical density, and is obtained with a carefully chosen asymptotics of the chemical potentials for the finite volume approximations. 

Recently, the program concerning the investigation of the spectral properties of the adjacency operator following the same lines as above, is carried out in \cite{Fidaleo} also for perturbed Cayley Trees.

To end the present introduction, we recall that the comb graphs considered here are a particular case of the graphs obtained via the comb product, and that, as studied in \cite{ABO, HO} and references therein, there is a correspondence between the notion of classical, resp. monotone, resp. Boolean, resp. free independence, and the notion of tensor, resp. comb, resp. star, resp. free, product of graphs.

\section{Geometrical Preliminaries}

 A {\it simple graph} $X=(VX,EX)$ is a collection $VX$ of objects,
 called {\it vertices}, and a collection $EX$ of unordered pairs of
 distinct vertices, called {\it edges}.  The edge $e=\set{u,v}$ is
 said to join the vertices $u,v$, while $u$ and $v$ are said to be
 {\it adjacent}, which is denoted $u\sim v$.

 Let us denote by $A=[A(v,w)]$, $v,w\in VX$, the adjacency
 matrix of $X$, that is,
 $$
 A(v,w)=|\set{v,w}|\,,\quad\set{v,w}\in EX
 $$ 
 and observe that, given $VX$, assigning $EX$ is equivalent to
 assigning $A$, that is the geometrical properties of $X$ can be
 expressed in terms of $A$.  For example, a graph is connected, namely
 any two vertices are joined by a path, is equivalent to the
 irreducibility of the matrix $A$,  the {\it degree} $\deg(v)$ of a
 vertex $v$, namely the number of vertices adjacent to $v$, counted with multiplicity, is equal
 to $(v,A^*Av)$ and, setting $d:=\sup_{v\in VX} \deg(v)$, we have
 $\sqrt{d}\leq\|A\|\leq d$, namely $A$ is bounded if and only if $X$
 has bounded degree.  We denote by $D=[D(v,w)]$ the degree matrix of
 $X$, that is,
$$
D(v,w)=
\begin{cases} 
    \deg v& v=w\\
    0&\text{otherwise.} 
\end{cases} 
$$  
The Laplacian on the graph is $\D=D-A$, so that 
$$
(\D f)(v) = \sum_{\set{v,w}\in EX}(f(w)-f(v))\,,
$$
 for any $f\in\ell^{2}(VX)$, $v\in VX$.\footnote{The definition used here implies $\D>0$, and differs from the standard one adopted in the physics literature.}

Assume now the simple graph $X$ to be countable and with bounded
 degree. In the present paper we only deal with bounded operators acting on $\ell^2(VX)$, if it is not otherwise specified.

 \begin{Dfn}\label{def:amenableGraph}
     Let $X$ be a countably infinite graph.  An increasing exhaustion
     $\set{K_n: n\in\bn}$ of finite subgraphs of $X$ is called an {\it
     amenable exhaustion} of $X$ if, setting $\cf K_{n}:= \set{v\in
     VK_{n} : d(v, VX\setminus VK_{n})=1}$, then $\displaystyle{
     \lim_{n\to\infty} \frac{|\cf K_{n}|}{|VK_{n}|} =0}$.

     $X$ is called an {\it amenable graph} if it possesses an amenable
     exhaustion.
 \end{Dfn}

We say that an operator $A$ acting on $\ell^2(VX)$ has {\it finite
 propagation} if there exists a constant $r=r(A)>0$ such that, for any
 $v\in X$, the support of $Av$ is contained in the (closed) ball
 $B(v,r)$ centered in $x$ and with radius $r$.  It is not difficult to
 show that finite propagation operators form a $^*$--algebra, and we denote
 by $\ca_{\text{FP}}(X)$ the generated C$^*$--algebra. 
 We say that a positive operator $T\in\ca_{\text{FP}}(X)$ is {\it
 essentially zero} if $\lim_n  \frac{Tr (TP_n)}{|VK_n|}=0$, where $P_n$ is the orthogonal projection onto the space generated by the vertices of $K_n$ in $\ell^2(VX)$.

 \begin{Prop}
     Essentially zero operators form the positive part of a closed
     two--sided ideal $\cai(X)$ of $\ca_{\text{FP}}(X)$.
 \end{Prop}
 \begin{proof}
     They clearly form a hereditary closed cone. We have to show that
     such cone is unitary invariant.  Indeed, if $B$ has finite
     propagation $r$, $P_n B = P_n B P_n(r)$, where $P_n(r)$ denotes
     the projection on the space generated by the vertices in
     $\cup_{v\in VK_n}B(v,r)$.  Therefore
     \begin{align*}
	 Tr \bigl( B^*TBP_n \bigr) &= Tr \bigl( B^*P_n(r)TP_n(r)BP_n\bigr) \leq \|B\|^2 Tr
	 \bigl( P_n(r)TP_n(r) \bigr)\\
	 &\leq \|B\|^2 Tr \bigl( P_nTP_n \bigr) + 2\|B\|^2 \|T\| Tr \bigl( P_n(r)-P_n \bigr).
     \end{align*}
     Moreover, $Tr \bigl( P_n(r)-P_n \bigr)$ can be estimated by the cardinality of
     $\cup_{v\in \cf K_n}B(v,r)$, hence by $|\cf K_n| d^{r+1}$. 
     Regularity of the exhaustion implies $\lim_n  \frac{Tr
     (B^*TBP_n)}{|VK_n|}=0$, namely $B^*TB$ is essentially zero.  Since
     the cone is closed, we get the invariance for unitaries in
     $\ca_{\text{FP}}(X)$.
\end{proof}

 Our first class of amenable graphs is given by the periodic ones. Let $\G$ be a countable discrete subgroup of
 automorphisms of $X$ acting freely on $X$ ($i.e.$ any $\g\in\G$,
 $\g\neq id$ doesn't have fixed points), and with finite quotient
 $B:=X/\G$.  Denote by $F\subset VX$ a set of representatives for
 $VX/\G$, the vertices of the quotient graph $B$. $F$ is called a {\it fundamental domain} for the periodic network $X$.

Let us define a unitary representation of $\G$ on $\ell^{2}(VX)$ by
 $(\l(\g)f)(x):= f(\g^{-1}x)$, for $\g\in\G$, $f\in\ell^{2}(VX)$,
 $x\in V(X)$.  Then the von Neumann algebra $\cn(X,\G):= \{ \l(\g) :
 \g\in\G\}'$, of bounded operators on $\ell^{2}(VX)$ commuting with the
 action of $\G$, inherits a trace given by $Tr_{\G}(T) = \sum_{x\in F}
 T(x,x)$, for $T\in\cn(X,\G)$. Clearly $A,D$ and $\D$ belong to $\cn(X,\G)$. The following theorem is known, see e.g. \cite{GILa03}, Theorem 6.2 for a proof.

 \begin{Thm}\label{thm:amenable}
     Let $X$ be a connected, countably infinite graph, $\G$ be a
     countable discrete amenable subgroup of automorphisms of $X$
     which acts on $X$ freely and cofinitely.  Then $X$ is an amenable
     graph, and $K_n$ can be chosen in such a way that, for a suitable
     choice of a sequence $E_n\subset \G$, $V(K_n\setminus\cf
     K_{n})\subseteq E_nF\subseteq V(K_n)$, namely $K_n$ is the finite
     union of copies of $F$ up to $\cf K_{n}$.
 \end{Thm}

 Let us now extend the
 previously defined trace.

 We first notice that, denoting by $P_n$ the projection on
 $\ell^2(VK_n)$, $Tr_\G(T)=\lim_n \frac{Tr (TP_n)}{|VK_n|}$, for any
 $T\in \cn(X,\G)$.  Indeed, by the properties above, the difference
 between the generic term of the sequence and  $Tr_\G(T)$ is
 infinitesimal.

Next Corollary immediately follows, cf. \cite{Ped}, Corollary 1.5.8.

\begin{Cor}\label{A(X)}
    Let us denote by $\ca(X)$ the space
    $\big(\cn(X,\G)\cap\ca_{\text{FP}}(X)\big)+ \cai(X)$.  Then, $\ca(X)$ is a
    C$^*$-algebra to which the trace $Tr_\G$ naturally extends.
\end{Cor}

\begin{Rem}
    All finite rank operators, and hence all compact operators, are
    essentially zero.  The trace $Tr_\G$ is faithful on $\cn(X,\G)$,
    but it is not faithful on $\ca(X)$, in particular it vanishes on
    compact operators.
\end{Rem} 

Now we discuss small perturbations of amenable periodic graphs.  If
$Y$ is a finite perturbation of an amenable periodic graph $X$, namely
they only differ for a finite number of points and edges, we can
consider both of them as subgraphs of a third graph $Z$, $X$ and $Y$
being obtained by removing finitely many vertices and finitely many
edges from $Z$.  It is not difficult to see that the exhaustion $K_n$ of $X$ can
be finitely perturbed to a amenable exhaustion $K'_n$ of $Z$, and that
$\cn(X,\G)$ is a (possibly non
unital) subalgebra of $B(\ell^2(VZ))$.

Reasoning as before, we can consider the unital $C^*$--algebra
$\ca(Z)=\big(\cn(X,\G)\cap\ca_{\text{FP}}(Z)\big)+\cai(Z)$, to which the trace
$Tr_\G$ naturally extends.  Since the adjacency operators $A_X$ and
$A_Y$ only differ for a finite rank operator, $A_Y\in\ca(Z)$, and $Tr_\G
(A^k_X)=Tr_\G (A^k_Y)$.  More generally, for any continuous function $\f$ on
$\br$, $Tr_\G \bigl(\f(A_X)\bigr)=Tr_\G \bigl(\f( A_Y)\bigr)$.
 This kind of invariance extends to a more general family of small
perturbations, which we call density zero perturbations.  

For the sake of simplicity, the result below concerns (possibly infinite) perturbations involving only edges. Of course, further finite perturbations can be treated as explained above. The
general case of density zero perturbations, studied in \cite{FGI}, can be recovered following the same lines.

\begin{Dfn}\label{essentiallyperiodic}
    Let $X$ be an amenable periodic graph, with exhaustion $K_n$, and
    consider a graph $Y$ such that $VX=VY$, so that $A_X$ and $A_Y$
    both act on the same Hilbert space $\ell^2(VX)$.  We say that $Y$
    is a {\it density zero} perturbation of $X$ if $A_X-A_Y$ is
    essentially zero.  In this case, $Y$ is also said to be an {\it
    essentially periodic} graph.
\end{Dfn}

\begin{Prop}\label{density0}
    Let $X$ be an amenable periodic graph, with exhaustion $K_n$, and
    let $Y$ be a graph with the same vertices as $X$.  Then $Y$ is a
    {\it density zero} perturbation of $X$ if and only if
    $$
    \lim_n\frac{|\set{u,v}\in EX\triangle EY\mid u\in VK_n\}|}{|VK_n|}=0
    $$
    where $EX\triangle EY$ denotes the symmetric difference. In this
    case, for any continuous function $\f$ on $\br$,
    \begin{equation}\label{invariance}
	Tr_\G \bigl( \f(A_X) \bigr) = Tr_\G \bigl( \f( A_Y) \bigr).
    \end{equation}
\end{Prop}
\begin{proof}
    Clearly $A_X-A_Y$ is essentially zero iff $\displaystyle \lim_n  \frac{Tr \bigl(
    (A_X-A_Y)^2P_n \bigr)}{|VK_n|}=0$.  A simple calculation shows that $\langle v,
    (A_X-A_Y)^2v \rangle = |\{e\in EX\triangle EY : v\in e\}|$. As a
    consequence, since edges for which both vertices are in $K_n$
    should be counted twice,
    $$
    |(EX\triangle EY)\cap EK_n|\leq Tr \bigl( (A_X-A_Y)^2P_n \bigr) \leq 2
    |(EX\triangle EY)\cap EK_n|.
    $$
    The thesis follows.  Concerning the last equality, setting
    $T=A_Y-A_X\in\cai(X)$, we have
    $A_Y^n-A_X^n=(A_X+T)^n-A_X^n\in\cai(X)$, namely (\ref{invariance})
    holds for $\f(t)=t^n$.  So the claim is true for any polynomial,
    and then, using Weierstrass density theorem, for any continuous
    function.
\end{proof}

\begin{Rem}
We note that for an essentially periodic graph $X$ the $C^*$-algebra $\ca(X)$, and the trace on it,  depend in principle on the exhaustion. However, previous Proposition implies that on geometric operators, such as the adjacency $A$ and its continuous functional calculi, the value of the trace is uniquely determined.
\end{Rem}

\section{Statistical mechanics on amenable graphs}

The main aim of the present paper is to investigate in full generality the thermodynamics of free Bosons (Baarden--Cooper pairs) on inhomogeneous networks with pure hopping Hamiltonian (i.e. the opposite of the adjacency matrix on the graph). Thus for the convenience of the reader, we report some standard notions useful in the sequel.

Let $(\ga,\a)$ be a dynamical system consisting of a (noncommutative) $C^*$--algebra and a one parameter group of $*$--automorphism $\a$. The state $\om\in\cs(\ga)$ 
satisfies the {\it KMS boundary condition} at inverse temperature $\b$, which we suppose to be 
always different from zero, if
\begin{itemize}
\item[(i)] $t\mapsto\om(A\a_{t}(B))$ is a continuous function for
every $A,B\in\ga$,
\item[(ii)]
$\int\om(A\a_{t}(B))f(t)dt=\int\om(\a_{t}(B)A)f(t+i\b)dt$
whenever $f\in\widehat{\cd}$, where `` $\widehat{}$ '' stands for the Fourier transform.
\end{itemize}
Here, $\cd$ is the space  of smooth compactly supported functions on $\br$.

The $C^*$--algebras considered here are those arising from the Canonical Commutation Relations (CCR for short). Namely, let $\gph$ be a pre-Hilbert space and consider the following (formal) relations between the annihilators $a(f)$, and creators $a^+(g)$, $f,g\in\gph$
\begin{equation}
\label{cccrr}
a(f)a^+(g)-a^+(g)a(f)=\langle f,g\rangle\,.
\end{equation}
It is well--known that the relations \eqref{cccrr} cannot be realized by bounded operators. A standard way to realize \eqref{cccrr} is to look at the symmetric Fock space $\cf_+(\bar\gph)$ on which the annihilators and creators naturally act as unbounded closed (mutually adjoint) operators. This concrete representation of the CCR is called {\it the Fock representation}. 

An equivalent description for the CCR is to put $\Phi(f):=\overline{a(f)+a^+(f)}/{\sqrt{2}}$, and define the {\it Weyl operators}
$W(f):=\exp{i\Phi(f)}$. The Weyl operators are unitary and satisfy the rule
\begin{equation*}
W(f)W(g)=e^{-i\frac{Im(f,g)}{2}}W(f+g)\,,\quad f,g\in\gph\,.
\end{equation*}
The CCR algebra $\ccr(\gph)$ is precisely the $C^*$--algebra generated by $\{W(f)\}_{f\in\gph}$.

Let $H$ be a positive operator acting on $\bar\gph$, and suppose that $e^{itH}\gph\subset\gph$. Then  
the one--parameter
group of Bogoliubov automorphisms $T_tf:=e^{itH}f$ defines a one--parameter group of $*$--automorphisms $\a_t$ of $\ccr(\gph)$ by putting $\a_t(W(f)):=W(e^{itH}f)$.

A representation $\pi$ of the CCR algebra $\ccr(\gph)$ is {\it regular} if
the unitary group $t\in\br\mapsto\pi(W(tf))$ is continuous in the
strong operator topology, for any $f\in \gph$.  A state $\f$ on $\ccr(\gph)$
is regular if the associated GNS representation is regular.  It simply
means that the functions $\{\f(W(tf))\}$ are continuous, for any $f\in
\gph$.
The {\it quasi--free} states of CCR algebras are of interest for our purposes. They are analytic states 
$\om$ uniquely determined by the two--point functions $\om(a^+(f)a(g))$, $f,g\in\gph$.

Let $X$ be an infinite graph, $\gph$ a subspace of $\ell^2(X)$, which contains the indicator functions of all finite
subregions $\La$ of the graph $X$.  A
representation $\pi$ of the CCR algebra $\ccr(\gph)$ is said to be {\it
locally normal (w.r.t. the Fock representation)} if
$\pi\lceil_{\ccr(\ell^2(\La))}$ is quasi--equivalent to the Fock
representation of $\ccr(\ell^2(\La))$.  A state on $\ccr(\gph)$ is locally
normal if the associated GNS representation is.  A locally normal
state $\f$ does have finite local density
$$
\r_{\La}(\f):=\frac{1}{|\La|}
\sum_{j\in\La}\f(a^{+}(\d_{j})a(\d_{j}))
$$
even if the mean density  might be infinite (e.g. if
$\lim_{\La\uparrow X}\r_\La(\f)=+\infty$). 

\begin{Lemma} \label{add1}
    If $\f$ is locally normal, then $\f\lceil_{\ccr(\ell^2(\La))}$ is regular
    for any finite subregion $\La$.
\end{Lemma}
\begin{proof}
    Since $\pi_\f\lceil_{\ccr(\ell^2(\La))}$ is quasi equivalent to the
    Fock representation of $\ccr(\ell^2(\La))$, they are unitary
    equivalent up to multiplicity (cf.  thm 2.4.26 \cite{BR1}).  The
    result follows since the Fock representation is regular.
\end{proof}

We now specialize  to the following situation.  Let $\La_{n}\uparrow X$
be a sequence of finite regions invading the graph $X$, together with
a sequence of states $\{\om_{\La_{n}}\}$ on $\ccr(\ell^2(\La_n))$ such that the following limit
$$
\lim_{n}\om_{\La_{n}}(a^{+}(v)a(v))=:q(v)
$$
exists (possibly $+\infty$) for each $v\in \gph$.

\begin{Lemma} \label{add2}
Suppose that    
\begin{equation*} 
\lim_{n}\om_{\La_{n}}(a^{+}(\d_{j})a(\d_{j}))=+\infty
    \end{equation*}
    for some $j\in X$. Then $\om(W(v)):=\lim_n\om_{\La_{n}}(W(v))$ does not define any locally normal state on $\ccr(\gph)$, where $\gph\subset\ell^2(X)$ is any subspace containing the finite supported sequences.
\end{Lemma}
\begin{proof}
    Let $j$ be contained in the finite region $\La$.  By Lemma
    \ref{add1}, it is enough to show that $\om\lceil_{\ga_{\La}}$ is
    not regular.  We have (cf. \cite{BR2}, Example 5.2.18)
  $$
	\om(W(\l\d_{j})) =\lim_{n}\om_{\La_{n}}(W(\l\d_{j}))
		= e^{-\frac14\l^2} \lim_n
	\exp\Bigl({-\frac{\l^2}2\om_{\La_{n}}(a^+(\d_j)a(\d_j))}\Bigr)=0\,.
   $$
    The thesis follows as $\om$ cannot be regular.
\end{proof}

For equivalent characterizations of the KMS boundary
condition and general results on the CCR the reader is referred to \cite{BR2} and the references cited therein.

 In this work, we consider amenable graphs (also called amenable
 networks) which are essentially-periodic, namely finite or density--zero
 perturbations of periodic graphs as described above.  We assume a
 regular exhaustion $\{\La_{n}\}$ is given, and denote by $\t$ the
 canonical trace on $\ca(X)$.  We also denote with an abuse of notation, $\ell^2(X)\equiv\ell^2(VX)$ for the network $X$. In this section we shall introduce the
 main thermodynamic properties and quantities.

 Fix a positive operator $H\in\ca(X)$ (the Hamiltonian) and denote
 by $N_{H}$ its {\it integrated density of states},  (see e.g. \cite{PF}) that is  $N_{H}(\l):=
 \t(E[0,\l])$, where $H=\int \l dE(\l)$ is the spectral decomposition
 of $H$.
 
 Let $\{H_{\La_{n}}\}$ be its finite volume truncation w.r.t. the
 exhaustion $\{\La_{n}\}$, $i.e.$ $H_{\La_n} := P_nHP_n$, and denote by $N_{H_{\La_n}}(\l) := \frac{1}{|\La_n|} Tr \bigl( E_{H_{\La_n}}[0,\l] \bigr) \equiv \frac{1}{|\La_n|} |\set{ t\in \s(H_{\La_n}) : t\leq \l}|$.  Define
 \begin{align} \label{a1}
    E_{0}(H):=&\lim_{\La_{n}\uparrow X}
    \bigg(\inf\supp\big(N_{H_{\La_{n}}}\big)\bigg)\,\\
     E_{m}(H):=&\inf\supp\bigg(\lim_{\La_{n}\uparrow X}
    N_{H_{\La_{n}}}\bigg) \equiv\inf\supp\big(N_{H}\big)\,.\nn
\end{align}
 
 \begin{Rem}
\itm{i} The limit in \eqref{a1} exists as a consequence of Lemma \ref{Lem:IncrSpectrum}.

\itm{ii} Because of Proposition \ref{traceconv}, we have $\lim_{\La_n \uparrow X} N_{H_{\La_n}} = N_H$, as distribution functions.

\itm{iii} If $A$ is the adjacency matrix of $X$, and $H:= \norm{A}-A$, then $E_{0}(H) =\lim_{\La_{n}\uparrow X}
     \norm{A} - \norm{A_{\La_{n}}} = 0$, while $E_m(H) = \norm{A} - \norm{\pi_\tau(A)}$, where $\pi_\tau$ is the GNS representation induced by $\tau$. Moreover, if $X$ is a periodic graph, then it is shown in Theorem \ref{NoHiddenForPeriodic} that $E_m(H)=0$.
\end{Rem}
 
 Obviously, $E_{0}(H)\leq E_{m}(H)$.  If $E_{0}(H)< E_{m}(H)$ we say
 that there is a {\it (low energy) hidden spectrum}, see e.g.
 \cite{BCRSV}. For the infinite graphs below, we shall always assume $E_{0}(H)=0$.

 \begin{Lemma} \label{Lem:IncrSpectrum}
     Let $X$ be a countable graph, $H\in\cb(\ell^{2}(X))$ a positive 
     operator, and let $\La_{1}\subset \La_{2}$ be finite subgraphs of 
     $X$. Then 
     $$
     \min \s(H_{\La_{1}}) \geq \min \s(H_{\La_{2}}).
     $$
 \end{Lemma}
 \begin{proof}
     Let $\l$ be the minimum eigenvalue of $H_{\La_{1}}$, and $v$ the
     relative normalised eigenvector.  Since $v\in \ell^{2}(\La_{1})$,
     we have $\langle v, H_{\La_{2}} v \rangle  = \langle v, H_{\La_{1}} v \rangle = \l$, that is
     $\l\in\set{ \langle w,H_{\La_{2}} w \rangle: w\in\ell^{2}(\La_{2}), \|w\|=1}$. 
     Observe that this set is contained in $\conv (\s(H_{\La_{2}}))$,
     since, if $H_{\La_{2}} = \int \l dE(\l)$ is the spectral
     decomposition, we have $\langle v,H_{\La_{2}} v \rangle = \int \l
     d\langle v,E(\l)v \rangle$, and the claim follows from the fact that
     $d\langle v,E(\l)v \rangle$ is a probability measure on $\s(H_{\La_{2}})$. 
     Finally, $\min \s(H_{\La_{2}}) = \min \conv \s(H_{\La_{2}})
     \leq \l = \min \s(H_{\La_{1}})$.
 \end{proof}
 
 \begin{Prop}\label{traceconv}
     Let $X$ be an essentially periodic graph with regular exhaustion
     $\set{\La_{n}}$, $H\in\ca(X)$ a positive operator. 
     Then, for any continuous function $\f:[0,\infty)\to\bc$, we have

     \itm{i} $\displaystyle\
     \lim_{n\to \infty} \frac{Tr(\f(H_{\La_{n}}))}{|\La_n|} = \t(\f(H))
     $,
     
     \itm{ii} $\displaystyle\ \int \f \,dN_H = \lim_{\La_n \uparrow X} \int \f  \,dN_{H_{\La_n}}$.
 \end{Prop}
 \begin{proof}
     Let us denote by $E_{n}$ the orthogonal projection from
     $\ell^{2}(X)$ onto $\ell^{2}(\La_{n})$.
     
     Observe that, for $k\geq 2$,
     \begin{align*}
	 Tr \bigl( E_{n}H^{k}E_{n} \bigr) &= Tr\bigl( E_{n} (
	 H(E_{n}+E_{n}^{\perp}) )^{k} E_{n} \bigr)\\
	 &= Tr\bigl( (E_{n}HE_{n})^{k} \bigr) + \sum_{ \substack{
	 \s\in\{-1,1\}^{k-1}\\ \s\neq \{1,1,\ldots,1\} } } Tr \bigl(
	 E_{n} \prod_{j=1}^{k-1} (HE_{n}^{\s_{j}})HE_{n} \bigr),
     \end{align*} 
     where $E_n^{-1}$ stands for $E_n^\perp$, and
     $$
     | Tr \bigl( E_{n} \prod_{j=1}^{k-1} (HE_{n}^{\s_{j}}) HE_{n}
     \bigr) | = | Tr \bigl( ...E_{n} HE_{n} ^\perp...  \bigr) | \leq
     \| H \|^{k-1} Tr ( |E_{n} HE_{n}^\perp|).
     $$
     Now assume $H$ has propagation $r$. Then, $$E_nHE_n^\perp =
     E_nE(B_{r}(V\La_n^c))HE_n^\perp=E(\La_n\cap B_{r}(V\La_n^c))HE_n^\perp,$$ hence
     $$Tr ( |E_{n} HE_{n}^\perp|) \leq \|H\|\ |B_{r-1}(\cf\La_n)|\leq\|H\|\ |\cf\La_n|d^{k-1},$$ where $d$ is the maximal degree of $X$.  As a
     consequence we obtain
     $$
     \frac1{|\La_n|}| Tr \bigl( E_{n}H^{k}E_{n} \bigr) - Tr\bigl(
     (E_{n}HE_{n})^{k} \bigr) | \leq (2d)^{k-1} \| H \|^{k}
     \frac{|\cf\La_n|}{|\La_n|} \to 0.
     $$
     Setting $\tau_n=\frac1{|\La_n|}Tr$, we have proved that $\lim_n\tau_n(E_np(H)E_n-p(E_nHE_n))=0$
     for any positive, finite propagation operator $H$ in $\ca(X)$. Since $|\tau_n(A)|\leq\|A\|$, the result follows by Weierstrass density theorem and the definition of $\ca(X)$.
 \end{proof}

 \begin{Dfn}\label{criticaldensity}
     Let $H$ be a positive operator in $\ca(X)$.  Then, for any {\it
     inverse temperature} $\b>0$, and for any {\it chemical potential}
     $\m\leq 0$, we
     define the {\it density} of $H$ as
     \begin{equation*} 
	 \r_{H}(\b,\m):=\int\frac{dN_{H}(h)}{e^{\b(h-\m)}-1}\,,
    \end{equation*}
     and the {\it critical density} of $H$ as
    \begin{equation} \label{dc}
	 \r^{H}_{c}(\b):=\int\frac{dN_{H}(h)}{e^{\b h}-1}\equiv \r_{H}(\b,0)\,.
     \end{equation}
     Let us recall that $H$  is called {\it recurrent} if the matrix elements $(\delta_x,H^{-1}\delta_x)$ are infinite, and {\it transient} if the matrix elements $(\delta_x,H^{-1}\delta_x)$ are finite.
 \end{Dfn}
 
 We say that BEC takes place for a given equilibrium state if a
 suitable portion of the particles occupies the lowest energy state.  
 \begin{Rem}
 \itm{i} It is also customary to fix the {\it  activity} $z:=e^{\b\m}$, instead of the chemical potential.
\itm{ii} If we choose $H$ as the graph Laplacian, transience and recurrence have a probabilistic interpretation. In fact, $(\delta_x,\D^{-1}\delta_y)=\sum_{n=0}^\infty p_n(x,y)$, where $p_n(x,y)$ is the probability of passing from $x$ to $y$ in $n$ steps, and the transition probability $p_1(x,y)$ is set to $(\deg x)^{-1}$ if $x$ and $y$ are adjacent, and to $0$ otherwise. As a consequence, if the graph is connected,  $(\delta_x,\Delta^{-1}\delta_x)$ is finite for all $x$ if and only if it is finite for a single $x$. Then recurrence corresponds to the following property of a random walk on $X$: the random walk starting at a point $x$ returns almost surely to $x$ infinitely many times.
Conversely,  a random walk is transient if the probability of  starting at  $x$ and returning to $x$ infinitely many times is zero. Interpreting a graph as an electrical network, transience means that the resistance between a point and infinity is finite.     For further results on transience and recurrence see e.g. \cite{Seneta}
 \itm{iii} For the standard homogeneous models investigated in literature (i.e.
 the statistical mechanics of free Bosons on $\br^{d}$ (cf. 
 \cite{BR2}), or on lattices with period $\bz^{d}$ (cf.  \cite{M}
 and the references cited therein)),\footnote{As noticed in \cite{BCRSV},  the BEC  behaviour of $\bz^{d}$--lattices only depends on $d$. This can be seen as the dispersion law ${\bf p}\mapsto\eps_{H}({\bf p})$ of a periodic Schr\"odinger operator $H:=\D+V$ on 
$\br^{d}$ or a lattice, has the same asymptotics near ${\bf 0}$
as that of the Laplacian $\D$, see e.g. \cite{RS}.}
   it is well--known that there exist
 equilibrium states exhibiting BEC if $\r^{H}_{c}(\b)<+\infty$.  This
 is also known to be equivalent to the transience of the graph, or to
 the fact that the growth of the graph is greater than 2.
 As we shall see in the following sections,  for the nonhomogeneous models treated in the present paper new phenomena (as for example the lack of the local normality of the resulting state in the thermodynamical limit) can happen.

 \itm{iv} Since we assumed $H$ to be bounded, the critical density is finite,
 namely the integral \eqref{dc} converges, iff
 \begin{equation*} 
    \int_{[0,\|H\|]}\frac{dN_{H}(h)}{h}<+\infty\,.
 \end{equation*}
 In particular, hidden spectrum implies finite critical density. 
A large class of examples of essentially periodic graphs exhibiting hidden spectrum will be described below.
\end{Rem}

Now we recall how equilibrium states are usually constructed.  Given a
positive Hamiltonian $H$ as above on the essentially periodic
graph $X$, one fixes an inverse temperature $\b>0$ and a density of particles $\r$, and determines the
chemical potential $\m(\La_{n})$ such that
 $$
 \r_{H_{\La_{n}}}(\b,\m(\La_{n}))=\r\,.
 $$
 To simplify the exposition, we suppose also that $\m(\La_{n})\to\m$, for some $\mu\in\br$, possibly
 by passing to a subsequence.  We necessarily have
 $\m(\La_{n})<E_{0}(H_{\La_n})$.  Since $E_{0}(H_{\La_n})\to
 E_{0}(H)=0$, we get $\m\leq0$.
 
 The finite volume state with density $\rho$ describing the Gibbs
 grand canonical ensemble in the volume $\La_{n}$ can be defined by
 the two--point function
\begin{equation} 
\label{sato}
     \om_{\La_{n}}(a^{+}(\xi)a(\eta))=\big\langle \eta, 
     (e^{\b(H_{\La_{n}}-\m(\La_{n})I)}-1)^{-1}\xi\big\rangle,
 \end{equation}
 where $\xi,\eta\in\ell^{2}(X)$.  Thermodynamical states are then
 described as limits of the finite volume states above.

 Let us now study the behavior of the density in the infinite volume limit.

 \begin{Prop} \label{ennz}
     Let $X$ be an essentially periodic graph with regular exhaustion
     $\set{\La_{n}}$, with positive Hamiltonian $H\in\ca(X)$, and
     assume $N_H(0)=0$.  Then, we have
     $$
     \lim_{\eps\downarrow0}\lim_{\La_{n}\uparrow X}
     \int\frac{f_{\eps}(h)}{e^{\b(h-\m(\La_{n}))}-1}dN_{H_{\La_{n}}}(h)
     =\r_H(\b,\m)\,,
     $$
     where $f_\eps$ is the continuous mollifier
     $$
     f_{\eps}(x)= 
     \begin{cases}
     0\,,&0\leq x\leq\eps\\
     \frac{x-\eps}{\eps}\,,&\eps<x\leq2\eps\\
     1\,,&2\eps<x\,.
     \end{cases}
     $$
 \end{Prop}
 \begin{proof}
     We compute
     \begin{align*}
	 &\bigg|\int\frac{f_{\eps}(h)}{e^{\b(h-\m(\La_{n}))}-1}
	 dN_{H_{\La_{n}}}(h)
	 -\int\frac{f_{\eps}(h)}{e^{\b(h-\m)}-1}dN_{H}(h)\bigg|\\
	 \leq&\bigg|\int\bigg(\frac{1}{e^{\b(h-\m(\La_{n}))}-1}
	 -\frac{1}{e^{\b(h-\m)}-1}\bigg)f_{\eps}(h)
	 dN_{H_{\La_{n}}}(h)\bigg|\\
	 +&\bigg|\int\frac{f_{\eps}(h)}{e^{\b(h-\m)}-1}dN_{H_{\La_{n}}}(h)
	 -\int\frac{f_{\eps}(h)}{e^{\b(h-\m)}-1}dN_{H}(h)\bigg|\to0,
     \end{align*}
     since
     $\frac{1}{e^{\b(h-\m(\La_{n}))}-1}\to\frac{1}{e^{\b(h-\m)}-1}$,
     uniformly on the support of $f_{\eps}$, and the second summand
     goes to zero because the finite volume sequences of traces
     converge to the infinite volume trace for fixed traceable
     operators, as follows from Proposition \ref{traceconv} $(ii)$.  The proof follows as $f_{\eps}\nearrow1$ whenever
     $\eps\searrow0$, essentially everywhere by taking into account that we
     assumed $N_H(0)=0$.\footnote{Recall that $\m=0$ is allowed
     and describes the most interesting situation of the BEC regime,
     see below.}
 \end{proof}

 Let us observe that the quantity
 $$
 n_{0}= n_{0}(\b,\r) :=\lim_{\eps\downarrow0}\lim_{\La_{n}\uparrow X}
 \int\frac{1-f_{\eps}(h)}{e^{\b(h-\m(\La_{n}))}-1}dN_{H_{\La_{n}}}(h)
 $$
 is well--defined and independent of the particular choice of the
 mollifier $f_{\eps}$ whenever the last converges monotonically to
 $\chi_{(0,+\infty)}$.  We have
\begin{equation*}
    \r=n_{0}+\r(\b,\m)\,.
\end{equation*}

From now on, we suppose that the Hamiltonians $H$, $H_{\La}$ are those
based on the adjacency matrix if it is not otherwise specified, and
drop some subscripts.  

\begin{Prop} Let $X$ be an essentially periodic graph, with adjacency matrix $A$ and $H= \norm{A}-A$. Then $N_H(0)=0$, i.e. the integrated  density of states is continuous in zero.
\end{Prop}
\begin{proof}
Let us assume first that $X$ is a periodic amenable graph.  Then, if $N_H(0)\ne0$, the $0$-eigenspace of $H$ is non-trivial, namely there is an $\ell^2$ Perron-Frobenius eigenvector for $A$, which is necessarily unique.  However, since $H$ is $G$ invariant, such vector is also periodic.  This is absurd. 
 
 Let now $X$ be essentially periodic. If there is no hidden spectrum, then $N_H(\lambda)$ is the same as that of a periodic graph, hence the result follows from the previous case. On the other hand, 
 hidden  spectrum immediately implies   $N_H(0)=0$.
 \end{proof}

 \begin{Lemma} \label{StrongConv}
     Let $X$ be an essentially periodic graph with regular exhaustion
     $\set{\La_{n}}$, $A$ the adjacency matrix of $X$.  Let $\set{\mu_n}\subset(-\infty,0]$, $\mu_n\to\mu$, and set
     $H:=(\|A\|-\mu)I-A$, and $H_{n}:= (\|A \|-\mu_n)I - A_{\La_{n}}$.  Let
     $\f:[0,\infty)\to\bc$ be a continuous function.  Then, for any
     $\xi\in\ell^{2}(X)$, $\f(H_{n})\xi\to \f(H)\xi$.
 \end{Lemma}
 \begin{proof}
     Let $x\in VX$; then, for $n\in\bn$ large enough,
     $A_{\La_{n}}\d_{x}=A\d_{x}$, so that $A_{\La_{n}}\xi=A\xi$, for
     $\xi\in\ell^{2}(X)$ and with finite support.  Let $\xi\in\ell^{2}(X)$
     be arbitrary, and, for any $\eps>0$, let
     $\xi_{\eps}\in\ell^{2}(X)$ and with finite support be such that
     $\|\xi-\xi_{\eps}\| <\eps$.  Then, for some $n_{\eps}\in\bn$, we
     have $A_{\La_{n}}\xi_{\eps}=A\xi_{\eps}$, for any $n>n_{\eps}$. 
     Therefore, for $n>n_{\eps}$ we have $\|A_{\La_{n}}\xi - A\xi\|
     \leq \|A_{\La_{n}}(\xi-\xi_{\eps})\| + \|A_{\La_{n}}\xi_{\eps} -
     A\xi_{\eps}\| + \|A(\xi-\xi_{\eps})\| \leq 2\|A\|
     \|\xi-\xi_{\eps}\| <2\|A\|\eps$.  Hence $H_{n}\xi\to H\xi$.
     
     Therefore, for any $k\in\bn$, $H_{n}^{k}\xi\to H^{k}\xi$, and the
     claim is true for any polynomial.  Finally, let
     $\f:[0,\infty)\to\bc$ be a continuous function, and, for any
     $\eps>0$, let $p$ be a polynomial such that $\|\f-p\|_\infty<\eps$. 
     Then, for any $\xi\in\ell^{2}(X)$, we have $\|\f(H_{n})\xi -
     \f(H)\xi\| \leq \|\f(H_{n})\xi - p(H_{n})\xi\| + \|p(H_{n})\xi -
     p(H)\xi\| + \|p(H)\xi - \f(H)\xi\| \leq 2\|\f-p\|_\infty \|\xi\| +
     \|p(H_{n})\xi - p(H)\xi\|$, from which the claim follows.
 \end{proof}
 
 \begin{Thm} \label{regi}
     Let $X$ be an essentially periodic graph with regular exhaustion
     $\set{\La_{n}}$, $A$ the adjacency matrix of $X$, $H:=\|A\|-A$.  Then
     
     \itm{i} $\m<0$ if and only if $\r<\r_{c}(\b)$,  
     
     \itm{ii} for any $\b>0$, $\m<0$, the sequence \eqref{sato}
     converges pointwise to a state $\om$, whose two--point function
     is given by
     $$
     \om(a^{+}(\xi)a(\eta))=\big\langle \eta, (e^{\b(H-\m
     I)}-1)^{-1}\xi \big\rangle\,.
     $$
     Moreover, the {\it density} $\r(\om)$ of the state $\om$, defined by 
    \begin{equation*} 
	 \r(\om):=\lim_{\La_{n}\uparrow X}\frac{1}{|\La_{n}|}
	\sum_{j\in\La_{n}}\om(a^{+}(\d_{j})a(\d_{j}))\,,
     \end{equation*}
     satisfies $\r(\om)=\r(\b,\m)$.     

     \itm{iii} The transience of $A$ is a necessary condition for
     the existence of locally normal states on $\ccr(\gph)$ at or above the critical
     density, {\it i.e.} for $\m=0$.  Here $\gph\subset\ell^2(X)$ is a subspace
     containing the functions with finite support.
 \end{Thm}
 \begin{proof}
     $(i) (\implies)$ If $\m(\La_{n})\to\m<0$, we can suppose that
     $z(\La_{n})\equiv e^{\b\m(\La_{n})}\leq K<1$ for each $n$.  We have
     \begin{align*}
	 &\int\frac{1-f_{\eps}(h)}{e^{\b(h-\m(\La_{n}))}-1}
	 dN_{H_{\La_{n}}}(h) \leq\frac{z(\La_{n})}{1-z(\La_{n})}
	 \int(1-f_{\eps}(h))dN_{H_{\La_{n}}}(h)\\
	 &\leq\frac{K}{1-K} \int(1-f_{\eps}(h))dN_{H_{\La_{n}}}(h)
	 \to\frac{K}{1-K}N_H(0) = 0\,,
     \end{align*}
     whenever $\La_{n}\uparrow X$, and then $\eps\searrow0$.  Therefore, $n_0=0$,
     so that $\r=\r(\b,\m)<\r(\b,0)\equiv\r_{c}(\b)$.

     \noindent $(\coimplies)$  Conversely, suppose that $\r<\r_{c}(\b)$ and $\m(\La_{n})\to0$. 
     Then
     \begin{align*}
	 \r_{c}(\b)& > \r=\int\frac{1-f_{\eps}(h)}{e^{\b(h-\m(\La_{n}))}-1}
	 dN_{H_{\La_{n}}}(h)
	 +\int\frac{f_{\eps}(h)}{e^{\b(h-\m(\La_{n}))}-1}dN_{H_{\La_{n}}}(h)\\
	 &\geq\int\frac{f_{\eps}(h)}{e^{\b(h-\m(\La_{n}))}-1}dN_{H_{\La_{n}}}(h)
	 \to\r(\b,0)\equiv\r_{c}(\b)\,,
     \end{align*}
     whenever $\La_{n}\uparrow X$, and then $\eps\searrow0$.  As this
     is impossible, we must have $\m<0$.

     $(ii)$ The fact that $\om_{\La_{n}}\to\om$ in the $*$--weak
     topology follows from $H_{\La_{n}}\to H$ in the strong operator
     topology.  Indeed, since $H_{n}-\m(\La_{n})\to H-\m$ strongly, it
     follows from Lemma \ref{StrongConv} that $\bigl( e^{\b(H_{n} -
     \m(\La_{n}))}-1 \bigr)^{-1} \to \bigl( e^{\b(H-\m)}-1
     \bigr)^{-1}$ strongly [because $\m<0$ so that $1\notin
     \s(e^{\b(H-\m)})$].  Therefore, for any $\xi,\eta\in\ell^{2}(X)$,
     we have $\om_{\La_{n}}(a^{+}(\xi)a(\eta)) = \langle \eta, \bigl(
     e^{\b(H_{n} - \m(\La_{n}))}-1 \bigr)^{-1} \xi  \rangle \to
     \langle \eta,  \bigl( e^{\b(H-\m)}-1 \bigr)^{-1} \xi \rangle =
     \om(a^{+}(\xi)a(\eta))$.  Since the vector space $V:=\ssv
     \set{ a^{+}(\xi)a(\eta) : \xi,\eta\in\ell^{2}(X)}$ is dense in
     $\ccr$, then $\om_{\La_{n}}$ is a Cauchy sequence in the
     weak$^{*}$ topology, so it converges.  Indeed, for any $T\in
     \ccr$, $\eps>0$, there is $T_{\eps}\in V$ such that
     $\|T-T_{\eps}\|<\eps$, so that, for any $m,n$ large enough,
     $|\om_{\La_{m}}(T)-\om_{\La_{n}}(T)| \leq
     |\om_{\La_{m}}(T)-\om_{\La_{m}}(T_{\eps})| +
     |\om_{\La_{m}}(T_{\eps})-\om_{\La_{n}}(T_{\eps})| +
     |\om_{\La_{n}}(T_{\eps})-\om_{\La_{n}}(T)| \leq 3\eps$.
         
     Finally,
     $$
     \int\frac{dN_{H_{\La_{n}}}(h)}{e^{\b(h-\m(\La_{n}))}-1}
     \to\int\frac{dN_{H}(h)}{e^{\b(h-\m)}-1}
     $$
     as in the proof of Proposition \ref{ennz}, because 
     $$
     \frac{1}{e^{\b(h-\m(\La_{n}))}-1} \to\frac{1}{e^{\b(h-\m)}-1}\,,
     $$
     uniformly on $[0,\|H\|]$, and the measures induced by the
     $N_{H_{\La_{n}}}$ converge in the $*$--weak topology to that
     induced by $N_{H}$.
     
     $(iii)$  Indeed, the following limit should be finite:
     \begin{align*}
	 & \lim_{\eps\downarrow 0}\lim_{\La_{n}\uparrow X} \langle
	 \xi,f_{\eps}(H_{n})(e^{\b(H_{n}-\m_{n})}-1)^{-1}\eta \rangle
	 \\
	 & = \lim_{\eps\downarrow 0}\lim_{\La_{n}\uparrow X}
	 \langle \xi,f_{\eps}(H_{n})(e^{\b H_{n}}-1)^{-1}\eta \rangle \\
	 & + \lim_{\eps\downarrow 0}\lim_{\La_{n}\uparrow X}\langle \xi,
	 f_{\eps}(H_{n}) \bigl( (e^{\b(H_{n}-\m_{n})}-1)^{-1}-(e^{\b
	 H_{n}}-1)^{-1} \bigr) \eta \rangle.
     \end{align*}
     Since for $\m_{n}\to 0$ the second summand is zero, we need the
     finiteness of the first.  By Lemma \ref{StrongConv},
     \begin{align*}
	 & \lim_{\eps\downarrow 0}\lim_{\La_{n}\uparrow X}
	 \langle \xi,f_{\eps}(H_{n})(e^{\b H_{n}}-1)^{-1}\eta \rangle \\
	 &= \lim_{\eps\downarrow 0} \langle \xi,f_{\eps}(H)(e^{\b
	 H}-1)^{-1}\eta \rangle \\
	 & = \langle \xi,(e^{\b H}-1)^{-1}\eta \rangle.
     \end{align*}
     The finiteness of $\langle \xi,(e^{\b H}-1)^{-1}\eta \rangle$
     when $\xi,\eta$ have finite support is exactly the transience of
     $A_X$.
 \end{proof}

 \begin{Rem}
  \itm{i} Let us observe that $n_{0}>0$ if and only if
  $\r>\r_{c}(\b)$.  Indeed, by definition, $n_0=0$ whenever $\m<0$,
  hence $n_0>0\Rightarrow \m=0$.  As a consequence,
  $0<n_0=\r-\r(\b,\m)=\r-\r_c(\b)$, i.e. $\r>\r_c(\b)$.  Conversely,
  $\r>\r_c(\b)$ implies $\m=0$, hence
  $n_0=\r-\r(\b,\m)=\r-\r_c(\b)>0$.
    
  \itm{ii} Observe that $n_{0}>0$ can be obtained only if
  $\m(\La_{n})\to0$.  Indeed, $\mu=0$ is a necessary condition for the
  occurrence of BEC.
\end{Rem}

\section{Some results on Perron--Frobenius eigenvectors}\label{PF}

Let $X$ be a finite connected graph, $A$ its adjacency matrix.  By
Perron--Frobenius Theorem there exists a unique eigenvector with
eigenvalue $\|A\|$, and it is the unique eigenvector having strictly
positive entries.
When $X$ is infinite, the existence of a square summable Perron--Frobenius
eigenvector is no longer guaranteed. If such vector exists is
unique and has strictly positive entries.
However, if $X$ has bounded degree, the equation $Av=\|A\|v$ makes
sense also for vectors which are simply functions $v:X\mapsto\bc$. 
Vectors of this kind satisfying the equation $Av=\|A\|v$ will be
called generalized Perron--Frobenius eigenvectors.  A generalized
Perron--Frobenius eigenvector has strictly positive entries but it is
not necessarily unique, see \cite{FGI} for the general Comb graphs, and \cite{Fidaleo} for Cayley Trees.  Indeed it is unique if the graph is
$A$--recurrent \cite{Seneta}.
Now we show the existence of such (generalized) Perron-Frobenius
eigenvectors for the network under consideration.

Assume $\La_n$ is an exhaustion for $X$, namely an increasing family
of connected finite subgraphs whose union is $X$, and choose a vertex
$x_0\in V\La_1$.  Then let $v_n$ be the Perron-Frobenius vector for
$A_n:=A_{\La_n}$ normalized by $\langle \d_{x_0},v_n \rangle = 1$.  We extend all
these vectors to zero  outside $\La_n$.

\begin{Prop}
With the above notation, $A$ has a generalized Perron--Frobenius eigenvector.
\end{Prop}
\begin{proof}
We first show that $\|A\|=\lim_n\|A_n\|$.  Indeed, if $Y\subset Z$ is a
proper inclusion of graphs with $Y$ finite and $Z$ connected, and $w$
is the norm-one Perron-Frobenius vector for $A_Y$, then
$\|A_Y\|= \langle w,A_Yw \rangle <  \langle w,A_Zw \rangle \leq\|A_Z\|$.  Therefore the sequence
$\|A_n\|$ is strictly increasing and bounded by $\|A\|$.  Assume now
ad absurdum that $\lim_n\|A_n\|<\|A\|$.  Then we could find a norm-one
vector $z$ with finite support such that $ \langle z,Az \rangle >\lim_n\|A_n\|$. 
Choosing $n$ large enough, we get $\langle z,Az \rangle =\langle z,A_nz \rangle \leq\|A_n\|$, which
gives a contradiction.

Finally we construct a generalized Perron-Frobenius vector for $X$. Since 
$$
1=\langle \d_{x_0}, v_n \rangle =\|A_n\|^{-1}\langle \d_{x_0},A_nv_n \rangle =\|A_n\|^{-1}\sum_{x\sim x_0}\langle \d_{x},v_n \rangle 
\geq \|A_n\|^{-1}\max_{x\sim x_0}\langle \d_x,v_n \rangle,
$$
then $x\sim x_0$ implies $\langle \d_{x},v_n \rangle \leq\|A_n\|$.  By induction we get
$d(x,x_0)\leq d$ implies $\langle \d_{x},v_n \rangle \leq\|A_n\|^d$.  >From this we can
obtain a subsequence $v_{n_k}$ such that $\langle \d_{x},v_{n_k} \rangle$ converges for
any $x\in X$.  Let us denote by $v$ the vector s.t. $\lim_k
\langle \d_{x},v_{n_k} \rangle =\langle \d_{x},v \rangle$.  We observe that, given $x\in X$, for sufficiently
large $k$, we have $A \d_{x}=A_{n_k} \d_{x}$.  Therefore
$$
\langle  \d_{x},Av \rangle =\langle A \d_{x},v \rangle = \lim_k \langle  \d_{x},A_{n_k}v_{n_k} \rangle =\lim_k
\|A_{n_k}\| \langle  \d_{x},v_{n_k} \rangle =\langle  \d_{x},  \|A\|v  \rangle,
$$
which means that $v$ is a generalized Perron-Frobenius vector for $X$.
\end{proof}

\section{No Hidden Spectrum for periodic graphs}

 In this section we want to show that there is no hidden spectrum for
 the adjacency operator on a periodic graph.  This means that the
 bottom of the spectrum for the energy operator $\|A\|-A$ coincides
 with the infimum of the support of the spectral measure of the energy
 operator in the trace representation, or, equivalently, that the
 supremum of the spectrum of $A$ coincides with the supremum of the
 spectrum for $\pi(A)$, where $\pi$ is the trace representation.

 We have already proved that, given an exhaustion $\La_{n}$ of the
 graph $X$, and denoting with $A_{n}$ the adjacency operator for
 $\La_{n}$, then $\|A_{n}\|\nearrow\|A\|$.

 Let now $v\in\ell^{2}(VX)$ be a unit vector with support contained in
 $\La_{n}$, $\g\in\G$ and consider the projection operator
 $P_{v}(\g)$ on the vector $\l(\g)v$.  Let us observe that, since
 $P_{v}(\g)$ is a projection operator on $\l(\g)v$, for any bounded
 operator $C$,
 $$
 Tr(CP_{v}(\g))=\langle \l(\g)v,C\l(\g)v \rangle.
 $$
 We need the following.

 \begin{Lemma}
     Let $X$ be a periodic graph, $K$ a finite subgraph, $v$ the
     normalised Perron-Frobenius eigenvector of $K$, and consider, for
     any $\g\in\G$, the projection operator $P_{v}(\g)$ on the vector
     $\l(\g)v$.  Then the series
     $\displaystyle{\sum_{\g\in\G}P_{v}(\g)}$ converges strongly to an
     operator $T$ which belongs to the von Neumann algebra
     $\cn(X,\G)$.
 \end{Lemma}
 \begin{proof}
     We need to show that, for any $\eps>0$, and $w=\sum_{y\in X}
     w_{y}\d_{y} \in\ell^{2}(X)$, there is a finite set
     $F_{\eps}\subset \G$ such that, for all finite sets $E\subset
     \G\meno F_{\eps}$, we have $\| \sum_{\g\in E} P(\g)w \|^{2}
     <\eps$.
     
     Indeed, for any finite $E\subset \G$, we have
     \begin{align*}
	 \|\sum_{\g\in E}P_{v}(\g)w\|^{2} & = \sum_{x\in X} | \langle
	 \d_{x},\sum_{\g\in E}P_{v}(\g)w \rangle |^{2} \\
	 & = \sum_{x\in X} | \sum_{\g\in E} \sum_{y\in X} 
	 w_{y} \langle \d_{x},P_{v}(\g)\d_{y} \rangle |^{2} \\
	 & = \sum_{x\in X} | \sum_{\g\in E} \sum_{y\in X} w_{y}
	 \langle \d_{x},\l(\g)v \rangle \langle \l(\g)v,\d_{y}\rangle
	 |^{2} \\
	 & \leq \sum_{x\in X} \Bigl( \sum_{\g\in E} \sum_{y\in X} |w_{y}| \cdot |
	 \langle \d_{x},\l(\g)v \rangle | \cdot | \langle \l(\g)v,\d_{y}
	 \rangle | \Bigr)^{2}.
    \end{align*}
    The scalar products in the last line are possibly non zero only 
    if $\g^{-1}x,\g^{-1}y\in K$, so that $d(x,y) = 
    d(\g^{-1}x,\g^{-1}y) \leq \diam K=:d$, and $x,y\in \g K\subset 
    EK := \cup_{\g'\in E} \g' K$. Let us observe that $|\set{\g\in E: 
    \g^{-1}x\in K}| = |\set{\g\in E: x\in\g K}| \leq |K|$, so that
    \begin{align*}
	\|\sum_{\g\in E}P_{v}(\g)w\|^{2} & \leq \sum_{x\in X} \Bigl(
	|\set{\g\in E: \g^{-1}x\in K}| \sum_{y\in B(x,d)\cap EK}
	|w_{y}| \Bigr)^{2} \\
	& \leq |K|^{2} \sum_{x\in X} \Bigl( \sum_{y\in B(x,d)\cap EK}
	|w_{y}| \Bigr)^{2} \\
	& \leq |K|^{2} \sum_{x\in X} |B(x,d)\cap EK| \sum_{y\in
	B(x,d)\cap EK} |w_{y}|^{2} \\
	& \leq |K|^{2} \sup_{x\in X} |B(x,d)| \sum_{y\in EK} \sum_{x\in
	B(y,d)} |w_{y}|^{2} \\
	& \leq |K|^{2} (\sup_{x\in X} |B(x,d)|)^{2} \sum_{y\in EK}
	|w_{y}|^{2}.
    \end{align*}
    
    Let now $H_{\eps}$ be a finite subset of $VX$ such that $\sum_{y\in 
    X\meno H_{\eps}} |w_{y}|^{2}<\eps$, and set $F_{\eps} := 
    \set{\g\in\G: \g K\cap H_{\eps} \neq\vuoto}$, so that, for any 
    finite set $E\subset \G\meno F_{\eps}$, we have $EK = \cup_{\g\in 
    E} \g K \subset \G\meno H_{\eps}$, and $\sum_{y\in EK} 
    |w_{y}|^{2}$, so that 
    $$
    \|\sum_{\g\in E}P_{v}(\g)w\|^{2}  \leq |K|^{2} (\sup_{x\in X} 
    |B(x,d)|)^{2} \eps,
    $$
    which establishes the claim.
 \end{proof}

 \begin{Thm} \label{NoHiddenForPeriodic}
    Let $A$ be the adjacency matrix of a periodic graph $X$, $\pi$ the
    trace representation of the von Neumann algebra $\cn(X,\G)$.  Then
    $$
    \|A\|=\sup\s(\pi(A))=\|\pi(A)\|,
    $$
    $i.e.$ $H:= \|A\| - A$ does not have hidden spectrum.
 \end{Thm}
 \begin{proof}
    Let $v_{n}$ be the (normalised) Perron-Frobenius vector for the
    restriction $A_{n}$ of $A$ to the graph $\La_{n}$, and let
    $P_{v_{n}},P_{v_{n}}(\g)$ be as defined above.  Moreover, for any $m>n$, let us
    denote by $E_{m}$ the projection on $\ell^{2}(\La_{m})$.  Then
    \begin{align*}
	Tr(E_{m}AP_{v_{n}})
	&=\sum_{\g\in\G}Tr(E_{m}AP_{v_{n}}(\g))\\
	&=\sum_{\g\in\G} \langle \l(\g)v_n, E_{m}A\l(\g)v_n \rangle \\
	&\geq\sum_{\g\in\G} \langle
	\l(\g)v_n, E_{m}\l(\g)E_{n}\l(\g)^{*}A\l(\g)v_n \rangle \\
	&=\|A_{n}\|\sum_{\g\in\G} \langle \l(\g)v_n, E_{m}\l(\g)v_n \rangle \\
	&=\|A_{n}\|\sum_{\g\in\G}Tr(E_{m}P_{v_{n}}(\g))
	=\|A_n\| Tr(E_{m}P_{v_{n}}),
    \end{align*}
    where the inequality follows by the positivity of all the entries,
    and the last but one equality follows from the fact that
    $E_{n}\l(\g)^{*}A\l(\g)E_{n}=A_{n}$.  As a consequence,
    $$
    \|A_{n}\|\leq\frac{Tr(E_{m}AP_{v_{n}})}{Tr(E_{m}P_{v_{n}})}
    \to\frac{\t(AP_{v_{n}})}{\t(P_{v_{n}})}\leq\sup\s(\pi(A)).
    $$
    where, for the last inequality, we used the following: let
    $\xi_{\t}$ be the GNS vector, and $\eta = \frac{\pi(P_{v_{n}})
    \xi_{\t}} {\| \pi(P_{v_{n}}) \xi_{\t}\|}$, then $\eta$ is normalised,
    and $\frac{\t(AP_{v_{n}})}{\t(P_{v_{n}})}= \langle \eta,\pi(A)\eta \rangle$.  The
    thesis follows.
 \end{proof}
 
 \begin{Prop}
     Let $X$ be a periodic graph. Then $H:= \|A\|-A$ has finite critical 
     density $\iff$ $A$ is transient.
 \end{Prop}
 \begin{proof}
     The critical density of $H$ is
     finite $\iff \t(H^{-1}) = \int_{0}^{\infty} \frac{dN_{H}(\l)}{\l}
     <\infty$.  Observe that $\t(H^{-1}) = \sum_{x\in F} \langle
     \d_{x},H^{-1}\d_{x} \rangle$, and recall that $\langle
     \d_{x},H^{-1}\d_{x} \rangle <\infty$ for some $x\in X$ $\iff
     \langle \d_{x},H^{-1}\d_{x} \rangle <\infty$ for all $x\in X \iff
     A$ is transient.
 \end{proof}

\section{The secular equation}
\label{secsec}

 Our aim here is to show that additive perturbations of an essentially
 periodic graph can lead to hidden spectrum for the adjacency matrix
 of the perturbed graph.  As explained above this is relevant for the
 occurrence of BEC condensation for the pure hopping model.  As a
 starting point for this analysis we write down an eigenvalue
 equation, called secular equation, for the adjacency matrix in
 terms of objects associated to the unperturbed graph. 
 
 We start by considering very general perturbations of the graph under consideration, then we specialize the matter to the case of interest for our purposes.
 Let $X,G$ be bounded degree graphs. We suppose we are adding and/or removing links from the graph $X$. Suppose further that $G$ is (possibly) attached to $X$, describing another perturbation of the latter. 
Let $Y$ be the resulting graph. Its adiacency matrix $A_{p}$ can be written as
\begin{equation} \label{1}
     A_{p}=
     \begin{pmatrix} 
	 A +D &C \\
	 C^{t} & B\\
     \end{pmatrix}\,,
 \end{equation} 
 where $A$, $B$ are the adjacency matrices of $X$, $G$ respectively,
 $D$ describes the contribution of the links added and/or removed from $X$, and finally $C$ describes the edges linking $G$ to $X$.  
 In the following $R_{T}(\l):=(\l I-T)^{-1}$ will denote the resolvent
 of the operator $T$ defined for $\l\not\in\s(T)$.  From now on we suppose that the resulting graph $Y$ is of bounded degree, this implies that $C$ and $D$ are bounded operators.
Define for $\l\not\in\s(A)\bigcup \s(B)$,
 \begin{equation*} 
    S(\l):= \big(DR_{A}(\l)+CR_{B}(\l)C^{t}R_{A}(\l)\big)
    \lceil_{\overline{\car(C)+\car(D)}}\,.
\end{equation*}

In the case under consideration,
 $S(\l)=PS(\l)P$ where $P$ is the orthogonal projection on the
closed subspace generated by the ranges of $C$ and $D$, on which we suppose $S$ naturally acts.
\begin{Thm} \label{eqsec}
With the above notation, $\l\not\in \s(A)\bigcup \s(B)$ is an eigenvalue of $A_{p}$ iff $1$ is
     an eigenvalue of $S(\l)$.  If this is the case, the corresponding
     eigenvectors $v:=\begin{pmatrix}x\\ y\\
     \end{pmatrix}$, respectively $z$, are related by
\begin{equation*}
z=Dx+Cy\,,
 \end{equation*}
     \begin{equation*}
	 x=R_{A}(\l)z\,,\quad y=R_{B}(\l)C^{t}R_{A}(\l)z\,.
     \end{equation*}
 \end{Thm}
 \begin{proof}
     Let $\l\not\in\s(A)\bigcup \s(B)$, and suppose there is $v\in\ell^2(Y)$ such that $A_{p}v=\l v$.  By \eqref{1}, we
     recover from the first equation
    \begin{equation} \label{3}
	 x=R_{A}(\l)(Dx+Cy)\,,
    \end{equation}
     and, by multiplying both sides by $D$,
    \begin{equation} \label{3a}
	 Dx=DR_{A}(\l)(Dx+Cy)\,.
    \end{equation}

     Analogously, from the second equation, we obtain
    \begin{equation*} 
	y=R_{B}(\l)C^{t}x\,,
    \end{equation*}
     and from \eqref{3}, by multiplying both sides by $C$,
     \begin{equation} \label{4a}
	 Cy=CR_{B}(\l)C^{t}x=CR_{B}(\l)C^{t}R_{A}(\l)(Dx+Cy)\,.
    \end{equation}

     Summing up \eqref{3a} and \eqref{4a}, we obtain that $z=Dx+Cy$ is
     an eigenvector of $S(\l)$ corresponding to the eigenvalue $1$. 
     Conversely, suppose that $z$ is an eigenvector of $S(\l)$ with
     eigenvalue $1$, and $\l\not\in \s(A)\bigcup \s(B)$.  Define
     $v:=\begin{pmatrix}R_{A}(\l)z\\ R_{B}(\l)C^{t}R_{A}(\l)z\\
     \end{pmatrix}$.  Then, it is easy to show that $v$ is an
     eigenvector of $A_{p}$ with eigenvalue $\l$.
 \end{proof}

 The equation
 \begin{equation} 
\label{5}
     [DR_{A}(\l)+CR_{B}(\l)C^{t}R_{A}(\l)]z=z\,.
 \end{equation}
 is called the {\it secular equation} in the present paper. It allows to compute the Perron--Frobenius eigenvalue of $A_p$ in many cases of interest, including some infinite, density zero, additive perturbations of periodic graphs.\footnote{Compare the computations in Section \ref{sec:CombGraphs} with those in \cite{BCRSV} used to prove the existence of the hidden spectrum for the comb graph.}
 
 Now we specialize the matter to the case of finite additive perturbations of a (essentially) periodic graph. In this case,  $B,C,D$ are finite rank operators, with $D$ positivity preserving, and acting on a finite dimensional subspace of $\ell^{2}(X)$. Thus, $S(\l)$ is a finite dimensional matrix whenever it is defined.
 
 We observe that in principle, $\|A_p\|$ might not be an eigenvalue of $A_p$, even if it is always the maximum of $\s(A_p)$ (cf. the existence of generalized Perron--Frobenius eigenvectors, see \cite{Seneta}). However,  if $\|A_p\| > \max\set{\|A\|,\|B\|}$, next result shows that $\|A_p\|$ is indeed an eigenvalue of $A_p$.
 
 \begin{Cor}
     Let $Y$ be a finite perturbation of $X$, and $A,B,C,D$ as above.  If
     $\|A_p\| > \max\set{\|A\|,\|B\|}$ then $1$ is an eigenvalue of
     $S(\|A_p\|)$ and $\|A_p\|\in\s_p(A_p)$.
 \end{Cor} 
 \begin{proof}
     Let $Y_n$ be an exhaustion of $Y$ such that $Y_1\supset
     G\cup\{x:\d_x\in\car(C)+\car(D)\}$,
     $\|A_{Y_1}\|>\max\set{\|A\|,\|B\|}$, and set $X_n=Y_n\cap X$.  By
     the results of Section \ref{PF} we may also assume that the
     Perron-Frobenius vectors $v_n=\begin{pmatrix}x_n\\
     y_n\\\end{pmatrix}$ for $A_{Y_n}$, normalized by taking value 1
     on a fixed vertex of $Y_1$, converge pointwise to a generalized
     Perron-Frobenius vector $v=\begin{pmatrix}x\\ y\\ \end{pmatrix}$
     for $A_Y$.  Applying Theorem \ref{eqsec} to any inclusion
     $X_n\subset Y_n$, we obtain that $z_n=Dx_n+Cy_n$ is an
     eigenvector with eigenvalue 1 for the matrix $S_n(\l_n):=
     \big(DR_{A_{X_n}}(\l_n)+CR_{B}(\l_n)C^{t}R_{A_{X_n}}(\l_n)\big)
     \lceil_{\car(C)+\car(D)}$, with $\l_n=\|A_{Y_n}\|$.  Let us
     observe
     that the vectors $z_n$ belong to the same finite-dimensional
     vector space, on which all the matrices $S_n(\l_n)$ act.  By
     construction, $\lim_n z_n=z:=Dx+Cy$ and $\lim_n S_n(\l_n)=
     S(\|A_Y\|)$, namely $z$ is an eigenvector with eigenvalue 1 of
     the matrix $S(\|A_Y\|)$.  Applying again Theorem \ref{eqsec} we
     show that $v$ is a true eigenvector of $A_Y$ with eigenvalue
     $\|A_Y\|$.
 \end{proof}

 We describe two particular cases of \eqref{5} when
 $D=0$,\footnote{Notice that the matrix $C^{t}R_{A}(\l)C$ is nonnull if
the graph $Y$ is supposed to be connected.}\label{eqsecfn}
\begin{equation} \label{5a}
    R_{B}(\l)C^{t}R_{A}(\l)Cy=y\,,\quad x=R_{A}(\l)Cy\,;
 \end{equation}
 and when $G=\emptyset$,
\begin{equation} \label{5b}
     DR_{A}(\l)z=z\,,\quad x=R_{A}(\l)z\,.
\end{equation}

 \begin{Cor}
     Let $Y$ be a finite perturbation of $X$, and $A,B,C,D$ as above. Assume that $\|B\|\geq\|A\|$, or $\|B\|<\|A\|$ and $A$ is recurrent. 
     Then $\|A_{p}\|>\max\set{\|A\|,\|B\|}$ and $\|A_{p}\|$ is an
     eigenvalue of $A_{p}$.
 \end{Cor}
 \begin{proof}
     The function $\l\in(\max\set{\|A\|,\|B\|},+\infty)\mapsto
     \|S(\l)\|$ is decreasing and tends to $0$ when $\l\to+\infty$. 
     When $\|B\| \geq \|A\|$, $\displaystyle
     \lim_{\l\to\|B\|}\|S(\l)\|=+\infty$.  When $\|B\|<\|A\|$,
     $\displaystyle \lim_{\l\to\|A\|}\|S(\l)\|=+\infty$ if and only if
     $A$ is recurrent.  So, in both cases, there exists a unique
     $\l_0>\max\set{\|A\|,\|B\|}$ such that $\|S(\l_0)\|=1$.  Since
     $S(\l_0)$ has positive entries, 1 is an eigenvalue, whose
     eigenvector $z$, the Perron-Frobenius eigenvector, has positive
     entries.  Applying Theorem \ref{eqsec} we get an eigenvector $v$
     of $A_p$, for the eigenvalue $\l_0$, having positive entries. 
     This immediately implies $\|A_p\|\geq\l_0>\max\set{\|A\|,\|B\|}$,
     therefore, by the previous Corollary, $\|A_p\|$ is an eigenvalue,
     whose eigenvector $v'$ has positive entries.  Then
     $\l_0=\|A_p\|$ and $v=v'$.
 \end{proof} 
 
 We end the present section by presenting a formula, which is needed
 in the  sequel, which describes $R_{A_{p}}(\l)$ in terms of the resolvents $R_A$ and $R_B$.
  
 \begin{Prop}\label{Resolvent}
 Let $Y$ be a finite perturbation of $X$, and $A,B,C,D, S(\l)$ as above. Consider $\l\in\bc$ such that
 $|\l|>\|A_{p}\|$, and choose $v=\begin{pmatrix}x\\ y\\  \end{pmatrix}$ in $\ell^2(Y)$. Then
 \begin{equation}\label{eq:Resolvent}
R_{A_{p}}(\l)v =
\begin{pmatrix}
R_{A}(\l)(x+z)\\
R_{B}(\l)(C^{t}R_{A}(\l)x+y+C^{t}R_{A}(\l)z)
\end{pmatrix},
 \end{equation}
 where $z=(I-S(\l))^{-1} \big( (DR_{A}(\l) + CR_{B}(\l)C^{t}R_{A}(\l))x + CR_{B}(\l)y \big)$.
 
 If $|\l|$ is sufficiently large, the formula holds also for infinite, additive perturbations with  density zero. 
 \end{Prop}
 \begin{proof}
 If we show that $I-S(\l)$ is invertible, the result follows from a straightforward calculation. If the perturbation is finite, $1\not\in\s_p(S(\l))$, otherwise, by Theorem \ref{eqsec}, $\l$ would belong to $\s_p(A_p)$, against the hypothesis $|\l|>\|A_{p}\|$. Since $S(\l)$ is a finite dimensional matrix, this means $1\not\in\s(S(\l))$, i.e. $I-S(\l)$ is invertible.  In the case of infinite perturbation, observe that $S(\l)\to0$ when $|\l|\to\infty$, and the thesis follows.
 \end{proof}
 
\begin{Rem}\label{SimpleCase}
Assume that the perturbation consists only of some extra edges, without adding vertices; then $C=B=0$, and the result above becomes
$$
R_{A_p}(\l)=R_A(\l)+R_A(\l)(I-DR_A(\l)|_{\car(D)})^{-1}DR_A(\l).
$$
In order to have the result for infinite perturbations it suffices that $|\l|>\|A\|+\|D\|$.
\end{Rem}

\section{Perturbations of periodic graphs}

The present section is devoted to some general results involving density zero perturbations of (essentially) periodic graphs.

 \begin{Prop} \label{smallPerturb}
 
     Let $X$ be an essentially periodic graph, and $Y$ a density zero perturbation
     of $X$.  Suppose that $\|A_{Y}\|>\|A_X\|$.  Then $Y$ has hidden
     spectrum.  In addition,
     \begin{equation} \label{rcfu}
	\r^Y_{c}(\b)=\r^{X}(\b,\m)
    \end{equation}
     where $\m=\|A_X\|-\|A_{Y}\|<0$.
 \end{Prop}
 \begin{proof}
     Let $g_{n}(x)$ be a continuous mollifier equal to $1$ if
     $x\leq\eps$ and $0$ if $x\geq\eps+1/n$.  We have, with $H:=
     \|A_{Y}\| - A_{Y}$ and $\eps<\|A_{Y}\|-\|A_X\|$,
     $$
     N_{H}([0,\eps])=\lim_{n}\t\big(g_{n}(\|A_{Y}\|-A_{Y})\big)
     =\lim_{n}\t\big(g_{n}(\|A_{Y}\|-A_X)\big)=0.
     $$
     As for the critical density, we have
     \begin{align*}	 
	 \r_{c}^{Y}(\b) & \equiv \t
	 \bigg(\big(e^{\b(\|A_{Y}\|I-A_{Y})}-1\big)^{-1} \bigg) =
	 \t\bigg(\big(e^{\b(\|A_{Y}\|I-A_X)}-1\big)^{-1}\bigg)\\
	 & = \t\bigg(\big(e^{\b(\|A_X\|I-A_X-\m I)}-1\big)^{-1}\bigg)
	 \equiv\r^{X}(\b,\m)\,.
     \end{align*}
 \end{proof}

Notice that \eqref{rcfu} allows us to compute the critical density of the perturbed network by using the formula for the density of the unperturbed one. It is very interesting for physical applications, to compare such a BEC critical density (equivalently critical temperature) with the critical density (temperature) of the formation of the Baarden--Cooper pairs in the Josephson junctions.

 We now consider finite subtractive perturbations of essentially periodic graphs.

 \begin{Thm} \label{subct}
     Let $Y$ be the graph obtained by removing a finite number of
     vertices and links from an essentially periodic graph $X$ which does not have hidden spectrum.  Then 
     
     \itm{i} $Y$ does not have hidden spectrum,
     
     \itm{ii} the critical densities of $X$ and $Y$ are equal.  
 \end{Thm}
 \begin{proof}
     $(i)$ Since $X$ does not have hidden spectrum, and finite perturbations do not change
     the trace $\t$, it suffices to show that $\|A_{Y}\|=\|A_{X}\|$. 
     It is known that $\|A_{Y}\|\leq\|A_{X}\|$.
     
     We obtain
     \begin{align*}
	 0&\leq \|A_{X}\|-\|A_{Y}\| \leq \|A_{X}\|-\|A_{Y}\| + 
	 E_{m}(\|A_{Y}\|-A_{Y})\\
	 & = E_{m}(\|A_{X}\|-A_{Y}) = 
	 E_{m}(\|A_{X}\|-A_{X}) = 0.
     \end{align*}
     Therefore $\|A_{Y}\|=\|A_{X}\|$ and $E_{m}(\|A_{Y}\|-A_{Y}) = 
     0 = E_{0}(\|A_{Y}\|-A_{Y})$, which is the claim.

     $(ii)$ Since $N_{Y}=N_{X}$, we obtain
     $\r_{Y}(\b,\m)=\r_{X}(\b,\m)$, and $\r^Y_c(\b)=\r^X_c(\b)$.
\end{proof}    

 Notice that Theorem \ref{subct} holds true for zero density
 subtractive perturbations.  In addition, it tells us that zero
 density subtractive perturbations do not alter the character of an
 essentially periodic graph, provided the graph under consideration does not exhibit hidden spectrum
 (e.g. a periodic graph).  Theorem \ref{subct} generalizes a result in
 \cite{M}.

 We now use the results of the previous section to show that very
 small additive perturbations of essentially periodic graphs provide
 examples of pure hopping low dimensional models with finite critical
 density.

 \begin{Prop} \label{Prop:pert}
     Let $X$ be an essentially-periodic graph with infinite critical
     density.  Then, there exists a point $x_{0}\in X$ such that if we
     add to $X$ only one vertex $\#$ linked to $x_{0}$, then the graph
     $X\cup\set{\#}$ has finite critical density.
 \end{Prop}
 
 \begin{proof}
 Denote by $A$ the adjacency matrix of $X$. The secular equation
     \eqref{5a} for $X\cup\{\#\}$ becomes $\l^{-1}\langle
     \d_{x_{0}}, R_{A}(\l)\d_{x_{0}} \rangle=1$, $x_{0}$ being the
     vertex (to be determined) of $X$ to which $\#$ is connected.  We
     will show that there exists $\l>\|A\|$ satisfying the previous
     equation.  
     
     Denote by $H=\|A\|-A$, and observe that $X$ has infinite critical
     density $\iff \infty = \int_{0}^{\|H\|} \frac{dN_{H}(\l)}{\l} =
     \int_{0}^{\infty} \frac{d\t(E_{H}(\l))}{\l} = \int_{0}^{\infty}
     \frac{d\t(E_{A}(\|A\|-\l))}{\l} = \int_{-\infty}^{\|A\|}
     \frac{d\t(E_{A}(\n))}{\|A\|-\n} \iff \int_{0}^{\|A\|}
     \frac{d\t(E_{A}(\l))}{\|A\|-\l} = \infty$.
     
     Therefore,
     \begin{align*}
	 \infty = & \int_0^{+\infty}\frac{dN_{A}(a)}{\|A\|-a}
	 = \int_0^{+\infty} \bigg(\sum_{n=0}^{+\infty}\|A\|^{-(n+1)}a^{n}\bigg)dN_{A}(a)\\
	 = & \sum_{n=0}^{+\infty}\|A\|^{-(n+1)}\int_0^{+\infty} a^{n}dN_{A}(a)
	 \equiv\sum_{n=0}^{+\infty}\|A\|^{-(n+1)}\t(A^{n})\\
	 \equiv&\sum_{n=0}^{+\infty}\|A\|^{-(n+1)}\lim_{\La_{k}\uparrow
	 X} \bigg(\frac{1}{|\La_{k}|}\sum_{x\in\La_{k}} \langle
	 \d_{x}, A^{n}\d_{x} \rangle\bigg)\\
	 \leq&\liminf_{\La_{k}\uparrow X}
	 \bigg(\frac{1}{|\La_{k}|}\sum_{x\in\La_{k}}
	 \bigg(\sum_{n=0}^{+\infty}\|A\|^{-(n+1)} \langle \d_{x},
	 A^{n}\d_{x} \rangle\bigg)\bigg)\,.
     \end{align*}

     Here, the second equality follows by the monotone convergence
     Theorem, and the last inequality by the Fatou Lemma.

     Then there exists $x_{0}\in X$ such that
     $$
     \langle \d_0, R_A(\|A\|)\d_0 \rangle =\sum_{n=0}^{+\infty}\|A\|^{-(n+1)} 
     \langle \d_{x_{0}}, A^{n}\d_{x_{0}} \rangle >\|A\|\,.
     $$

     This means that, for the decreasing function $\l^{-1}\langle \d_{x_{0}},
     R_{A}(\l)\d_{x_{0}} \rangle$,
     \begin{align*}
	 &\lim_{\l\to+\infty}\l^{-1}\langle \d_{x_{0}},
	 R_{A}(\l)\d_{x_{0}} \rangle=0\,,\\
	 &\lim_{\l\to\|A\|}\l^{-1} \langle\d_{x_{0}}, 
	 R_{A}(\l)\d_{x_{0}} \rangle>1\,.
     \end{align*}

     Namely, there exists a (unique) $\l>\|A\|$ such that the secular
     equation for $X\cup\{\#\}$ is satisfied, or by Theorem
     \ref{eqsec}, there is a (unique) $\l>\|A\|$ which is an
     eigenvalue of $A_{p}$, which implies $\|A_{p}\|\geq \l > \|A\|$. 
     Therefore, $E_{m}(\|A_{p}\|-A_{p}) = E_{m}(\|A_{p}\|-A) =
     \|A_{p}\|-\|A\| + E_{m}(\|A\|-A) \geq \|A_{p}\|-\|A\|>0$, that is
     $X\cup\{\#\}$ exhibits low energy hidden spectrum.
 \end{proof}

 \begin{Thm} \label{pert}
     Let $X$ be a essentially-periodic graph.  Then, there exists a point
     $x_{0}\in X$ such that if we add to $X$ only one vertex $\#$
     linked to $x_{0}$, then the graph $Y:=X\cup\set{\#}$ verifies
     $\|A_Y\|>\|A_X\|$.  This implies $Y$ has hidden spectrum, hence
     finite critical density.
 \end{Thm}
 \begin{proof}
     Assume the critical density of $X$ is finite. If 
     $\|A_{X}\| = \|A_{X\cup\set{\#}}\|$, then the critical densities 
     of $X$ and $X\cup\set{\#}$ are equal, so that  $X\cup\set{\#}$ has 
     finite critical density. If $\|A_{X}\| < \|A_{X\cup\set{\#}}\|$, 
     then $X\cup\set{\#}$ has hidden spectrum and finite critical density.
     
     If $X$ has infinite critical density the result follows from the
     Proposition above.
 \end{proof}
 
 \begin{Rem}\label{Rem:nail}
 \item{$(i)$} The proof of Proposition \ref{Prop:pert}
 is based on the fact that there exists $x_0$ such that $\langle
 \d_{x_0}, R_A(\|A\|)\d_{x_0} \rangle$ is large enough.  This is
 trivially true if $X$ is periodic, since in that case infinite
 critical density is equivalent to recurrence, namely $\langle \d_{x},
 R_A(\|A\|)\d_{x} \rangle =+\infty$ for any $x$.\footnote{See \cite{ABO} for results related to the BEC and the computation of  the "vacuum distribution" $d\m(\l)=\langle\d_0, dE_A(\l)\d_0 \rangle$ in some cases of interest in quantum probability, such as the comb graph.}
 
 \item{$(ii)$}
 Notice that, the divergence of the same integral
 $\int\frac{dN_{A}(a)}{\|A\|-a}$, on the one hand is responsible of
 the infinite critical density for the unperturbed graph $X$, on the
 other hand allows us to conclude that graphs obtained by considering
 very small additive perturbations of $X$ have finite critical
 density. 
 \end{Rem}

 Up to now we have shown that small perturbations of a graph can produce hidden spectrum to the pure hopping Hamiltonian. The remarkable fact, pointed out in the following theorem, is that such a network cannot exhibit hidden spectrum if we choose as Hamiltonian of the model the Laplace operator $\D$ of the graph.

 \begin{Thm} \label{lapl}
     Let $X$ be a periodic amenable graph, and $Y$ a density zero
     perturbation of $X$.  Let the Hamiltonian $H$ be the Laplace
     operator on $Y$.  Then $H$ does not have hidden spectrum.
 \end{Thm}
 \begin{proof}
     Let $\D$, and $\D_{p}$ be the Laplacian of $X$, and the
     perturbed graph $Y$, respectively.  We have
     $$
     0\leq E_{0}(\D_{p})\leq E_{m}(\D_{p})\equiv E_{m}(\D)=0\,,
     $$
     where the last equality follows by Theorem 2.55, (5) of \cite{Luck}.
 \end{proof}

\section{One dimensional examples}\label{ex}
 
 In this section, we exhibit some examples of graphs which have hidden spectrum. To prove that, we use proposition \ref{smallPerturb}, so we have to compare the norms of the adjacency operators of a graph and its perturbation. To compute the norms, we use Perron-Frobenius theory, and in particular the secular equation \eqref{5}. We start by considering the eigenvalue equation for the adjacency operator
 on a linear chain.  As we shall see, this gives rise to a difference
 equation whose solutions form a 2-dimensional space.  Therefore two
 more data, such as the value on boundary points, determine the
 solution for the given eigenvalue, and a further datum determines the
 eigenvalue.  In this way we can calculate the Perron-Frobenius
 eigenvector and eigenvalue on one-sided or two-sided linear chains
 with perturbations.  Also, we can compute the matrix elements of the
 resolvent $R_A(\l)$ for one-sided or two-sided linear chains. 
 Indeed, the vector $v=R_A(\l)\d_x$ satisfies the equation
 $(\l-A)v=\d_x$, namely the eigenvector equation with a perturbation. 
 We compute some examples below.
 
 \begin{Exmp} [Modified chain graphs]
 
 Suppose that the perturbation is on the left of a linear chain, and
 it determines the first components, denoted by $(\a_{0},\b_{0})$,  of an eigenvector corresponding to an
 eigenvalue $\l>2$ for the adjacency matrix.  Denote the other
 components on the right as $(\a_{1},\b_{1},\a_{2},\b_{2},\dots)$, see
 figure \ref{Fig1}.  The remaining components on the right are the
 solution of the finite--difference system
 \begin{align} \label{fds}
     \begin{pmatrix} 
	 \a_{n+1}\\
	 \b_{n+1}\\
     \end{pmatrix}
     =&    
     \begin{pmatrix} 
	 -1&\l\\
	 -\l& \l^{2}-1\\
     \end{pmatrix}
     \begin{pmatrix} 
	 \a_{n}\\
	 \b_{n}\\
     \end{pmatrix}.	
 \end{align}

     \begin{figure}[ht]
 	 \centering
	 \psfig{file=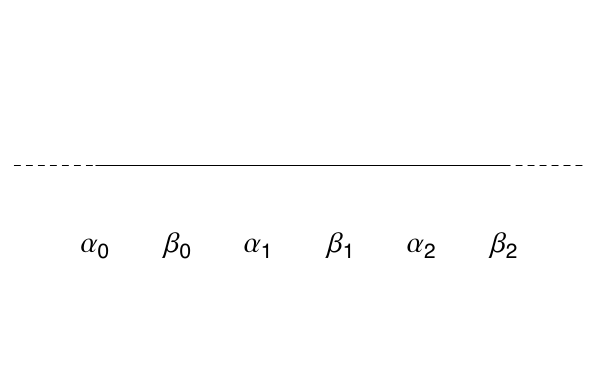,height=1.5in}
	 \caption{Infinite chain.}
	 \label{Fig1}
     \end{figure}

 The eigenvalues of the matrix in \eqref{fds} are
\begin{equation*} 
     \mu_{\pm}(\l)=\frac{\l^{2}-2\pm\l\sqrt{\l^{2}-4}}{2}\,,
 \end{equation*}
 with corresponding eigenvectors
 \begin{equation*} 
    v_\pm(\l):=
 \begin{pmatrix} 
    2\\
    \l\pm\sqrt{\l^{2}-4}
 \end{pmatrix}\,.
 \end{equation*}

 Now we apply the previous considerations to the graph $X$ in figure \ref{Fig2}.  Since the square-summable Perron-Frobenius eigenvector is unique, it is necessarily symmetric. So, we only search for symmetric eigenvectors of $A_X$.
   We have, in the previous
 notation,
\begin{equation}\label{eq1}
\l=\a_{0}\,,\quad\l\a_{0}=1+2\b_{0}\,.
 \end{equation}

 By taking into account that, in order to get a square summable vector
 on the chain, $
 \begin{pmatrix} 
     \a_{0}\\
     \b_{0}\\
 \end{pmatrix}$ cannot have a component along the eigenvector
 $v_{+}(\l)$, we obtain
 \begin{equation}\label{eq2}
 \begin{pmatrix} 
     \a_{0}\\
     \b_{0}\\
 \end{pmatrix}
 =a
 \begin{pmatrix} 
     2\\
     \l-\sqrt{\l^{2}-4}\\
 \end{pmatrix}\,.
 \end{equation}

 Solving \eqref{eq1} and \eqref{eq2} w.r.t. $\l$, we obtain $\l=\sqrt{2+\sqrt{5}}>2$. Since the eigenvector we found has only positive components, it is the Perron-Frobenius eigenvector of $A_X$. Therefore $\norm{A_X} = \l > 2 = \norm{A_\bz}$, and, by proposition \ref{smallPerturb}, the graph $X$ has hidden spectrum.
 
      \begin{figure}[ht]
 	 \centering
	 \psfig{file=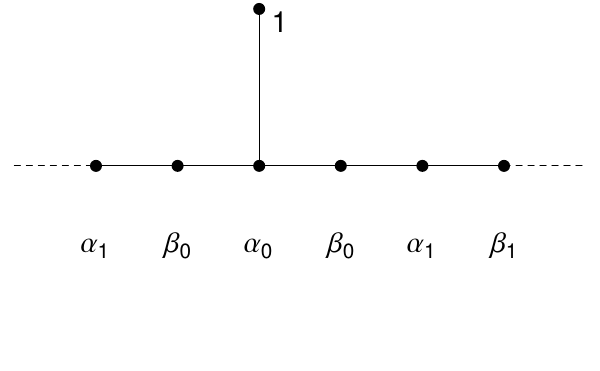,height=1.5in}
	 \caption{Infinite chain with a nail.}
	 \label{Fig2}
     \end{figure}

 Other examples can be solved along the same lines, as  those in figure \ref{Fig3}. They both have hidden spectrum.

      \begin{figure}[ht]
 	 \centering
	 \psfig{file=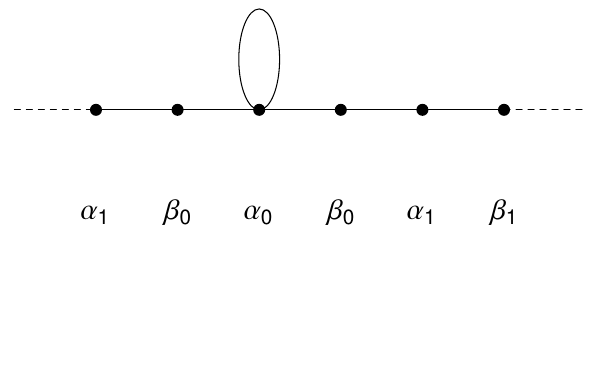,height=1.2in} \qquad \qquad
	 \psfig{file=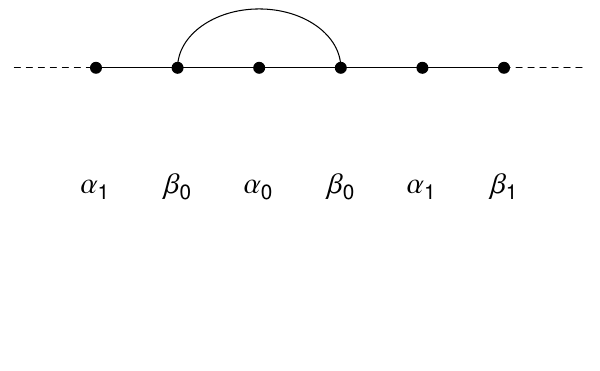,height=1.2in}
	 \caption{Some infinite graphs.}
	 \label{Fig3}
     \end{figure}
\end{Exmp}

 Another application of the previous method is to compute some matrix
 elements which are used in the sequel.

 \begin{Prop}
     We have for the following matrix elements,
     \begin{align} 
	 \langle \d_{0}, R_{A_{\bn}}(\l)\d_{0} \rangle =
	 \frac{2}{\l+\sqrt{\l^{2}-4}}\,, \label{mel}\\
	 \langle \d_{0}, R_{A_{box}}(\l)\d_{0} \rangle =
	 \frac{2}{\l+\sqrt{\l^{2}-8}}\, \label{mel1}
     \end{align}
     where $A_{box}$ is the adjacency matrix of the box graph in
     figure \ref{Fig4}.
 \end{Prop}
 \begin{proof}
     We compute the latter, the computation of the former being similar.  The one
     dimensional dynamical system (as that given in \eqref{fds})
     associated to the box graph is described by the matrix
     $$   
     \begin{pmatrix} 
	 -1&\l\\
	 -\frac{\l}{2}& \frac{\l^{2}}{2}-1\\
     \end{pmatrix}
     $$
     whose eigenvector associated to the negative eigenvalue has the form
     $$   
     \begin{pmatrix} 
	 4\\
	 \l-\sqrt{\l^{2}-8}\\
     \end{pmatrix}\,.
     $$

     We have 
     $$
     \begin{pmatrix} 
	 \a_{0}\\
	 \b_{0}\\
     \end{pmatrix}
     =a
     \begin{pmatrix} 
	 4\\
	 \l-\sqrt{\l^{2}-8}\\
     \end{pmatrix}\,,\quad\l\a_{0}-2\b_{0}=1.
     $$
     Solving w.r.t $\a_{0}\equiv\langle\d_{0},
     R_{A_{box}}(\l)\d_{0} \rangle$, provides the assertion.
 
     \begin{figure}[ht]
 	 \centering
	 \psfig{file=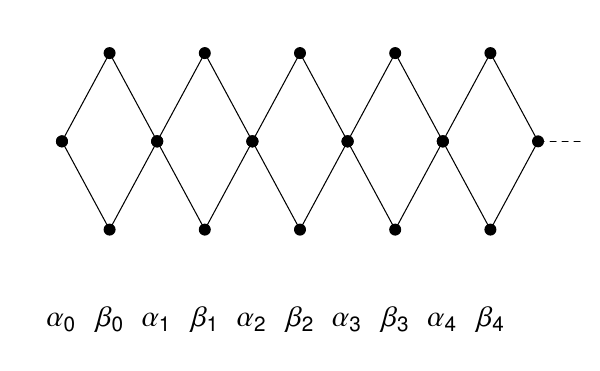,height=1.5in} 
	 \caption{Box graph.}
	 \label{Fig4}
     \end{figure}
 \end{proof}

\begin{Rem}
 By applying the same calculation as before, we obtain
 \begin{equation} \label{vzeta}
     \langle \d_{0}, R_{A_{\bz}}(\l)\d_{0}
    \rangle=\frac{1}{\sqrt{\l^{2}-4}}\,.
 \end{equation}
\end{Rem}

 Now we apply the previous results to compute the Perron--Frobenius
 eigenvalue and/or eigenvector of some pivotal examples in order to
 show that they exhibit low dimensional hidden spectrum.
 
 \begin{Rem}\label{N}
     Even though $\bn$ is not a finite perturbation of a periodic
     graph, the disjoint union of two copies of $\bn$, $\bn \sqcup
     \bn$, can be identified with the graph $\bz$ with one link
     removed.  The embedding of $\bn \sqcup \bn$ in $\bz$ gives rise
     to an embedding of pairs $(T,S)$ of operators in $\ca_{FP}(\bn)$
     into operators in $\ca_{FP}(\bz)$.  We may therefore define the
     C$^*$-algebra $\ca(\bn)$ as that consisting of the operators
     $T\in\ca_{FP}(\bn)$ such that $(T,T)\in\ca(\bz)$, endowed with
     the corresponding trace. This simple observation allows us to conclude that the critical density 
     $\r_c(\b)$ of the graph $\bn$, resp. the box--graph, is infinite as it coincides with that of $\bz$, resp. the bilateral box--graph (which is a $\bz$--lattice).
 \end{Rem}

 \begin{Exmp}[Star and star-box graphs]
     The star graph with $n\geq 3$ strands is composed by $n$ copies of
     $\bn$ all connected to a single vertex $\#$, see figure
     \ref{Fig5}, left.  If $\l>2 = \norm{A_\bn}$, the secular equation \eqref{5a} is written as
     \begin{equation} \label{c2}
	 \l=n\langle \d_{0}, R_{A_{\bn}}(\l)\d_{0} \rangle\,.
    \end{equation}

     By taking into account \eqref{mel}, we obtain for the adjacency
     matrix $A_{n}$ of the star graph with $n$ strands,
     $$
     \l\equiv\|A_{n}\|=\frac{n}{\sqrt{n-1}}\,.
     $$

     Therefore, the star graph has hidden spectrum,
     see \cite{BCRSV}.

    The star--box graph is made of $n\geq 3$
     copies of the box
    graph connected to a single vertex $\#$, see figure \ref{Fig5},
    right.  By taking into account \eqref{c2} and \eqref{mel1}, we have
    $\l=\frac{2n}{\l+\sqrt{\l^{2}-8}}$ which gives
    $$
    \l=\frac{n}{\sqrt{n-2}}\,, \quad n\geq 4.
    $$

    Since $\norm{A_{box}} = 2\sqrt{2}$,  the star--box graph has hidden spectrum
    iff $n\geq5$.

      \begin{figure}[ht]
 	 \centering
	 \psfig{file=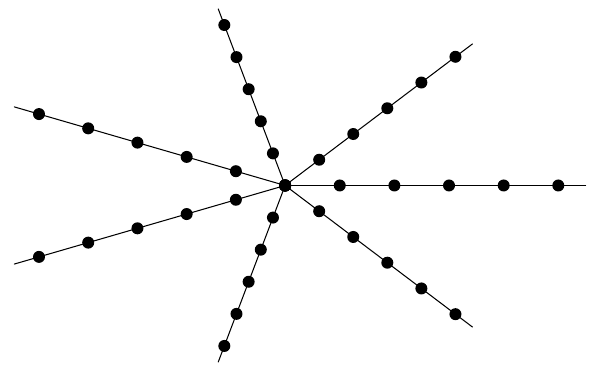,height=1.2in} \qquad \qquad
	 \psfig{file=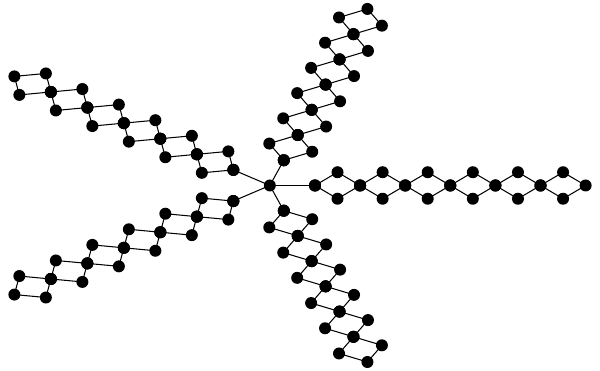,height=1.2in}
	 \caption{Star and star-box graphs.}
	 \label{Fig5}
     \end{figure}
\end{Exmp}

\begin{Exmp}[Polygonal star and star-box graphs]

    We now apply the secular equation \eqref{5b} to solve the star--box
    graph with $n$ strands obtained by connecting $n\geq 3$ copies of the box
    graph through a polygon, see figure \ref{Fig6}, right.  The
    vector $z$ in \eqref{5b} is supported on the vertices of the
    polygon having $n$ edges.  If $\l> \norm{A_{box}} =2\sqrt{2}$, the secular equation for such a 
    $z$
    living on the polygon, becomes
   \begin{equation} \label{c1}
	\langle \d_{0}, R_{A_{box}}(\l)\d_{0} \rangle A_{pol}z=z\,,
   \end{equation}
    where $A_{box}$, $A_{pol}$ are the adjacency matrices of the box
    graph and the polygon respectively.  Rotational invariance for the graph under consideration implies $z$ has equal components, hence
    $2\langle \d_{0}, R_{A_{box}}(\l)\d_{0} \rangle=1$, independently
    of the number of the edges of the polygon.  By taking into
    account \eqref{mel1}, the previous equation gives $\l=3>2\sqrt{2}\equiv\|A_{box}\|$. 
    Hence, the star--box graph has hidden spectrum.

     Another simple example of the star graph is that made of $n$
     strands connected by a polygon, see figure \ref{Fig6}, left. 
     By taking into account \eqref{c1} and \eqref{mel}, we have
     $\frac{4}{\l+\sqrt{\l^{2}-4}}=1$.  Namely,
     $\l=\frac{5}{2}>2\equiv\|A_{\bn}\|$, so that the graph has hidden
     spectrum.

       \begin{figure}[ht]
 	 \centering
	 \psfig{file=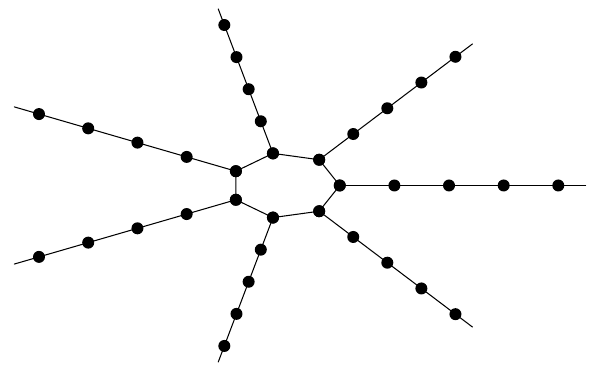,height=1.2in} \qquad \qquad
	 \psfig{file=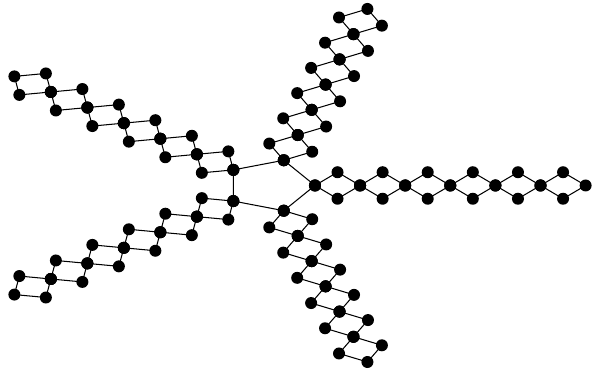,height=1.2in}
	 \caption{Polygonal star and star-box graphs.}
	 \label{Fig6}
     \end{figure}
\end{Exmp}

\begin{Exmp}[H-graphs]
We consider two copies of the bilateral infinite chain to which we
    add $k$ links between the two origins,  see figure \ref{Fig7a}. We call this graph an $H$-graph.  We
    have
    $$
    \a_{0}
    \begin{pmatrix} 
	2\\
	\l-k\\
    \end{pmatrix}
    =a\a_{0}
    \begin{pmatrix} 
	2\\
	\l+\sqrt{\l^{2}-4}\\
    \end{pmatrix}
    +b\a_{0}
    \begin{pmatrix} 
	2\\
	\l-\sqrt{\l^{2}-4}\\
    \end{pmatrix}\,.
    $$
    In order to have a square-summable eigenvector, we need $a=0$, from which we obtain
    $$
    \l=\sqrt{k^2+4}.
    $$
    Therefore, the H-graph has hidden spectrum as soon as $k>0$.
       \begin{figure}[ht]
 	 \centering
	 \psfig{file=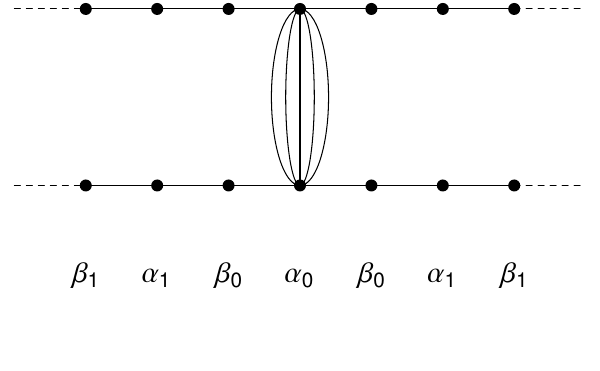,height=1.5in} 
	 \caption{H graphs.}
	 \label{Fig7a}
     \end{figure}
     \end{Exmp} 

\begin{Exmp}[Modified ladder graphs]
    In the previous examples, we considered additive and subtractive
    perturbations separately.  We now consider them together.

    We consider the bilateral ladder graph modified as follows.  We
    add $k-1$ links at the origin, and remove $2n$ links
    symmetrically, see figure \ref{Fig7b}. We look at a modified ladder graph as a graph containing a suitable $H$--graph as a subgraph.

 Since this graph contains the H-graph as a subgraph, $\norm{A_X}\geq \sqrt{k^2+4}$. Since the ladder graph ($i.e.$ for $k=1$ and $n=0$) has  $\norm{A_{ladder}} = 3$, the modified ladder graph $X$ has hidden spectrum for all $k\geq 3$, $n \geq 0$. Moreover, by theorem \ref{subct} it follows that  $X$ has no hidden spectrum for $k=0$ and any $n\geq 0$, or for $k=1$ and any $n\geq 1$. Finally, it is possible to prove that, for $k=2$, $X$ has hidden spectrum for $n=0$, and no hidden spectrum for $n\geq 1$.
 \begin{figure}[ht]
 	 \centering
	 \psfig{file=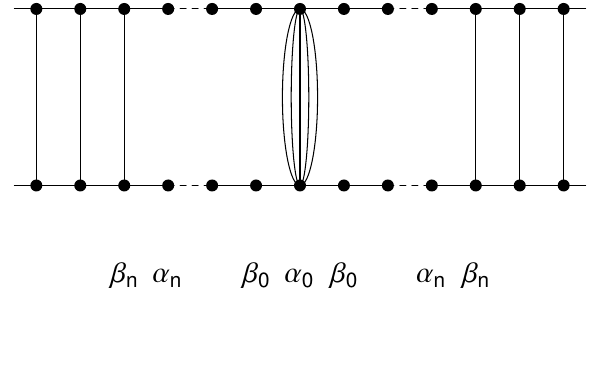,height=1.5in} 
	 \caption{Modified ladder graphs.}
	 \label{Fig7b}
     \end{figure} 
\end{Exmp}

\section{Comb graphs}\label{sec:CombGraphs}

 In \cite{BCRSV} the authors considered a graph, which they called
 comb graph, and showed that it has low energy hidden
 spectrum.

 In general, one can define the comb product between two graphs, as in
 the following definition (cf. \cite{ABO}).
 
 \begin{Dfn}
     Let $G$, $H$ be graphs, and let $o\in VH$ be a given vertex. 
     Then the comb product $X := G \comb (H,o)$ is a graph with $VX :=
     VG \times VH$, and $(g,h)$, $(g',h') \in VX$ are adjacent $iff$
     $g=g'$ and $h\sim h'$ or $h=h'=o$ and $g\sim g'$.  We call $G$ the {\it base graph}, and
     $H$ the {\it fibre graph}.  When $o\in H$ is understood from the
     context, we omit it, and write $G\comb H$.
 \end{Dfn}

 In this and the next section we shall consider the so called comb
 graphs, with base $\bz^d$, $d\in\bn$, and fibre $\bz$, with
 distinguished vertex $0\in\bz$, and we denote them simply by $\bz^d
 \comb \bz$.  As we shall see, the comb graphs exhibit a different
 behaviour with respect to BEC, if $d\leq 2$ or $d\geq 3$.

 The comb graph $\bz\comb \bz$ can be described as an additive
 perturbation of the disconnected graph given by $\bz$ copies of $\bz$
 (i.e. the fibres).  The perturbation consists of adding some extra
 links: for any $n\in\bz$, there is a link connecting the zero--point
 of the $n$-th copy to the zero--point of the $(n+1)$-th copy, see
 figure \ref{Fig8}.  The added links form a copy of $\bz$ which is
 usually called the backbone in the Physics literature.  We endow this
 graph with the regular exhaustion $\{\La_n\}_{n\in\bn}$, where
 $\La_n$ is the square $[-n,n]\times[-n,n]$. Here, and in the following, we denote by $[m,n] := \set{z\in \bz: m\leq z \leq n}$.
   In this sense the comb
 graph is a density zero perturbation of the disconnected graph given
 by infinitely many disjoint copies of $\bz$, indeed, according to
 Proposition \ref{density0}, it is sufficient to note that
   $$
   \frac{|(E(\bz\comb\bz)\setminus E(\sqcup_\bz\bz))\cap
   E\La_n|}{|V\La_n|}=\frac{2n+1}{(2n+1)^2}\to0.
   $$

 As already said, it was shown in \cite{BCRSV} that the comb graph $\bz\comb\bz$
 has low energy hidden spectrum.  We generalize this result
 as follows.
 
 \begin{Prop}
     The comb graph $G_{d}=\bz^d\comb\bz$ has low energy hidden
     spectrum.  In particular $\|A_{G_{d}}\|=2\sqrt{d^{2}+1}$.  The
     generalized Perron-Frobenius vector on $G_d$, obtained as
     point-wise limit of Perron-Frobenius vectors on $\La_n$, can be
     explicitly calculated.
 \end{Prop}
 
 \begin{proof}
     Let us observe that $\La_n$ can be described as a finite comb
     graph $[-n,n]^d\comb[-n,n]$, hence as a finite perturbation of
     the disjoint union of $(2n+1)^{d}$ copies of $[-n,n]$.  Applying
     Theorem \ref{eqsec} with $\l_n=\|A_{[-n,n]^d\comb[-n,n]}\|$, and in
     particular equation \eqref{5b}, we obtain the equation for the  Perron-Frobenius 
     eigenvector $z_n$ for the matrix $S(\l_n)$ corresponding to the eigenvalue 1, given by
   \begin{equation*}
   \label{dzdn}
	 \langle \d_{0}, R_{A_{[-n,n]}}(\l_n)\d_{0} \rangle
	A_{[-n,n]^{d}}z_n=z_n.
    \end{equation*}
     This means 
     $R_{A_{[-n,n]}}(\l_n)\d_{0} \rangle \|A_{[-n,n]^{d}}\|=1$.
    Taking the limit for $n\to\infty$, we obtain, with $\l:=\lim \l_n$,
     $$
     \langle \d_{0}, R_{A_{\bz}}(\l)\d_{0} \rangle \|A_{\bz^{d}}\|=1
     $$
     which, according to (\ref{vzeta}), leads to
     \begin{equation*} 
	 \frac{2d}{\sqrt{\l^{2}-4}}=1\,,
    \end{equation*}
     or equivalently
     $$
     \l=\|A_{G_d}\|=2\sqrt{d^{2}+1}\,.
     $$
     Notice that $z_n$ is the Perron-Frobenius vector for
     $A_{[-n,n]^{d}}$, and we can normalize $z_n$ in order to to have
     value 1 in the origin.  Then $z_n$ converges pointwise to the
     unique generalized Perron-Frobenius vector for $\bz^d$, namely
     the vector which is constantly equal to 1 on $\bz^d$.  Then,
     according to \eqref{5b}, the Perron-Frobenius vector $x_n$ of
     $A_{[-n,n]^d\comb[-n,n]}$ is given by
     $x_n=R_{A_{[-n,n]}}(\l_n)z_n$.  As a consequence, for any $(\vec\jmath,j)\in V(\bz^d\comb \bz)$,
     $$
     \langle \d_{\vec{\jmath},j},x_n \rangle = \langle
     \d_{\vec{\jmath},j},R_{A_{[-n,n]}}(\l_n)z_n \rangle = \langle
     \d_{\vec{\jmath}},z_n \rangle \langle \d_j,R_{A_{[-n,n]}}(\l_n)\d_0
     \rangle.
     $$
     Taking the limit for $n\to\infty$ we show that $x_n$ converges
     pointwise; its limit is the generalized Perron-Frobenius vector
     $x$ whose component $(\vec{\jmath},j)$ is equal to $\langle
     \d_j,R_{A_{Z}}(\l)\d_0 \rangle$, with $\l=2\sqrt{d^{2}+1}$.
 \end{proof}
 
Let us consider the comb graph $G^{d}=\bz^{d} \comb \bz$ together with the finite volume approximations $\La_{n}=X_n\comb Y_n$, where $X_n$ is the graph $(\bz_{2n+1})^{d}$ (periodic boundary condition on the base graph), and $Y_n$ is the finite chain $[-n,n]$.

\begin{Lemma}\label{tensorPF}
Describing $\ell^2(\La_n)$ as $\ell^2(X_n)\otimes\ell^2(Y_n)$, the Perron Frobenius eigenvector $v_n$ for the adjacency operator $A_{\La_n}$ has the form $u_n\otimes R_{Y_{n}}(\|A_{\La_n}\|)\d_{0}$, where  $u_n$ is the vector constantly equal to 1 on $X_n$. Moreover,
\begin{equation*}
 2d\langle\d_0, R_{Y_{n}}(\|A_{\La_n}\|)\d_{0}\rangle=1
\end{equation*}
\end{Lemma}
\begin{proof}
Indeed,
     $A_{\La_n} = I\otimes A_{Y_{n}} + A_{X_n} \otimes P_{0}$, $P_0$ denoting the one-dimensional projection on $\d_0$, so that, for $v_n=u_n\otimes w_n$,
\begin{align*}
	A_{\La_n}v_{n}  = u_{n} \otimes A_{Y_{n}} w_{n} + A_{X_n} u_{n}
     \otimes P_{0}w_{n}  
    = u_{n}
     \otimes (A_{Y_n}w_{n} + 2d P_{0}w_{n}),
\end{align*}
where we used the fact that the constant vector on $X_n$ is the Perron-Frobenius vector for $A_{X_n}$, and the equality $\norm{A_{X_n}}=2d$.

Then, $v_n$ is an eigenvector for $A_{\La_n}$ with eigenvalue $t$ if $A_{Y_n}w_{n} + 2d P_{0}w_{n}-
tw_n=0$, which gives 
\begin{equation}\label{PFeq}
(t - A_{Y_{n}})w_{n} =  2d\langle \d_{0},w_{n} \rangle \d_{0}.
\end{equation}
 In particular, this implies
\begin{equation}\label{extracond}
2d\langle\d_0, R_{Y_{n}}(t)\d_{0}\rangle=1,
\end{equation}
 where $R_{Y_{n}}(t)$ denotes the resolvent $(t-A_{Y_{n}})^{-1}$.  Let us note that the function $\langle\d_0, R_{Y_{n}}(t)\d_{0}\rangle$ is decreasing in $(\|A_{Y_n}\|,+\infty)$, $\displaystyle\lim_{t\to\|A_{Y_n}\|}\langle\d_0, R_{Y_{n}}(t)\d_{0}\rangle=+\infty$ and $\displaystyle\lim_{t\to+\infty}\langle\d_0, R_{Y_{n}}(t)\d_{0}\rangle=0$, therefore there exists a $t$ for which condition (\ref{extracond}) is satisfied. With such a $t$,  the choice $w_n=R_{Y_{n}}(t)\d_{0}$ gives rise to an eigenvector $v_n$ by equation (\ref{PFeq}). Moreover such $v_n$ has positive entries, hence is the Perron-Frobenius vector. This implies that the $t$ satisfying equation (\ref{extracond}) coincides with $\|A_{\La_n}\|$, so it is unique.
\end{proof}

We may now use the preceeding Lemma to obtain results on the graph $G^d$.

\begin{Lemma}\label{convergence}
Let $v_n=u_n\otimes w_n$, with $u_n$ constantly equal to 1 on the base $X_n$, $w_n=\|R_{Y_{n}}(\|A_{\La_n}\|)\d_{0}\|^{-1}R_{Y_{n}}(\|A_{\La_n}\|)\d_{0}$, and $v=u\otimes w$, with $u$  constantly equal to 1 on the base graph $\bz^d$ and $w = \|R_{\bz}(\|A\|)\d_{0}\|^{-1}R_{\bz}(\|A\|)\d_{0}$. Then
\item{$(i)$} $w_n$ converges in norm to $w$,
\item{$(ii)$} $v$ is a generalized Perron-Frobenius vector for $A$, 
\item{$(iii)$} $\langle\d_0, R_{Y_{n}}(\l)\d_{0}\rangle=\frac{\tanh (n+1)\th}{\sqrt{\l^2-4}}$, where $2\cosh\th=\l$,
\item{$(iv)$} $\|R_{\bz}(\|A\|)\d_{0}\|^2=\frac{\sqrt{d^2+1}}{4d^3}$.
\end{Lemma}
\begin{proof}
$(i)$ Let us observe that $\|A_{\La_n}\|\to\|A\|=2\sqrt{d^2+1}$, with $A=A_{G^d}$,  
while $\|A_{Y_n}\|\to\|A_{\bz}\|=2$, 
hence Lemma \ref{StrongConv} implies that $R_{Y_{n}}(\|A_{\La_n}\|)$ converges strongly to $R_{\bz}(\|A\|)$.\\
$(ii)$ The proof is analogous to the result in Section \ref{PF}.\\
$(iii)$ Setting $\l=2\cosh\th$, it is not difficult to check that the vector $z(\l,n)$ defined by
\begin{equation*}
z(\l,n)_j=\frac{\sinh[(n+1-|j|)\th]}{2\sinh\th\cosh[(n+1)\th]}=
\tanh[(n+1)\th]\frac{\cosh(|j|\th)}{2\sinh\th}-
\frac{\sinh(|j|\th)}{2\sinh\th},\quad |j|\leq n,
\end{equation*}
satisfies $(\l I-A_{Y_n})z(\l,n)=\d_0$, therefore $\langle\d_0, R_{Y_{n}}(\l)\d_{0}\rangle=z(\l,n)_0$. The thesis follows since $2\sinh\th=2\sqrt{\cosh^2\th-1}=\sqrt{\l^2-4}.$
\\
$(iv)$ By $(i)$, $z(\|A\|,n)$ converges in norm to $R_{\bz}(\|A\|)\d_{0}$, therefore
\[
\langle\d_j,R_{\bz}(\|A\|)\d_{0}\rangle=\frac{e^{-|j|\th}}{2\sinh\th},
\]
with $2\cosh\th=\|A\|=2\sqrt{d^2+1}$. The thesis follows by a straightforward computation.
\end{proof}

Let us set $v=\sum_{\vec\jmath}\d_{\vec\jmath}\otimes v_{\vec\jmath}$, $v_{\vec\jmath}$ denoting the restriction of $v$ to the fibre at the point ${\vec\jmath}$.
We say that $v\in\cs_{0}$ if  the sequence $\|v_{\vec\jmath}\|$, ${\vec\jmath}\in\bz^d$, is rapidly decreasing.
We now show that $T_{t}\equiv e^{itH}$ defines a one--parameter
group of Bogoliubov automorphisms on $\cs_{0}$.

\begin{Prop}
\label{add4}
Let $H=\|A\|-A$ on the comb graphs $G_{d}$. Then
$e^{itH}\cs_{0}\subset\cs_{0}$.
\end{Prop}
\begin{proof}
If $v\in\ell^{2}(G_{d})$ then
$$
e^{itH}v=\frac{1}{2\pi i}\oint_{\g}e^{it(\|A\|-\l)}R_{G^d}(\l)v\, d\l
$$
where $\g$ is a Jordan curve surrounding counterclockwise the spectrum
of $H$, and $R_{G^d}(\l)=(\l-A)^{-1}$ is the resolvent of the adjacency operator on the comb graph. 
Let $(\vec\jmath,j)=(j_{1},\dots,j_{d},i)$ denote the coordinates of
the comb graph $G_{d}$, and denote by 
$\d_{\vec\jmath}\otimes\d_j$ the delta function on a point $({\vec\jmath},j)$. Let us recall that, by Proposition \ref{Resolvent}, for $|\l|$ large enough,
$R_{G^d}(\l) = I \otimes R_{\bz}(\l) + \Phi(\l) \otimes R_{\bz}(\l) P_{0} R_{\bz}(\l)$
with
\begin{equation*}
\Phi(\l):=s(\l)R_{\bz^{d}}(s(\l))A_{\bz^{d}}\,,
\end{equation*}
and $s(\l)$ is the holomorphic extension of $\sqrt{\l^{2}-4}$ to
$\bc\backslash[-2,2]$. 

Let us set $v=\sum_{\vec\jmath}\d_{\vec\jmath}\otimes v_{\vec\jmath}$, $v_{\vec\jmath}$ denoting the restriction of $v$ to the fibre at the point ${\vec\jmath}$.
Then the assumption $v\in\cs_0$ amounts to say that the sequence $\|v_{\vec\jmath}\|$, ${\vec\jmath}\in\bz^d$, is rapidly decreasing, while the thesis, namely $e^{itH}v\in\cs_0$, is equivalent to say that the sequence $\sum_{j\in\bz}|\langle\d_{\vec\jmath}\otimes\d_j, e^{itH}v\rangle |^2$, ${\vec\jmath}\in\bz^d$, is rapidly decreasing.
By the equations above, we have
\begin{align*}
\langle\d_{\vec\jmath}\otimes\d_j,&R_{G^d}(\l)v\rangle=\\
&=\langle\d_j,R_\bz(\l)v_{\vec\jmath}\rangle+\sum_{{\vec k}\in\bz^d}\langle\d_{\vec\jmath},\Phi(\l)\d_{\vec k}\rangle\ \langle\d_j,R_{\bz}(\l) P_{0} R_{\bz}(\l)v_{\vec k}\rangle.
\end{align*}
Choose the curve $\g$ as a circle with radius greater than 
$\|A_\bz\|+\|A_{\bz^d}\|=2(d+1)$, which surrounds the spectrum of $A_{G_d}$, and the hypothesis of Proposition \ref{Resolvent} are satisfied, see Remark \ref{SimpleCase}. With such a choice,
\begin{align*}
\sum_{j\in\bz}&|\langle\d_{\vec\jmath}\otimes\d_j, e^{itH}v\rangle |^2
\leq\frac{\ell(\g)^2}{2\pi^2}\sup_{\l\in\g}|e^{-it\l}|^2\|R_\bz(\l)\|^2 \|v_{\vec\jmath}\|^2+\\
&+\frac{\ell(\g)^2}{2\pi^2}\sup_{\l\in\g}|e^{-it\l}|^2\sum_{j\in\bz}\left|\sum_{{\vec k}\in\bz^d}\langle\d_{\vec\jmath},\Phi(\l)\d_{\vec k}\rangle\ \langle\d_j,R_{\bz}(\l) P_{0} R_{\bz}(\l)v_{\vec k}\rangle\right|^2.
\end{align*}
Clearly the first summand is rapidly decreasing in ${\vec\jmath}$. Concerning the second summand, we have
\begin{align*}
\sup_{\l\in\g}&\sum_{j\in\bz}\left|\sum_{{\vec k}\in\bz^d}\langle\d_{\vec\jmath},\Phi(\l)\d_{\vec k}\rangle\ \langle\d_j,R_{\bz}(\l) P_{0} R_{\bz}(\l)v_{\vec k}\rangle\right|^2=\\
&=\sup_{\l\in\g}\left\|\sum_{{\vec k}\in\bz^d}\langle\d_{\vec\jmath},\Phi(\l)\d_{\vec k}\rangle\ R_{\bz}(\l) P_{0} R_{\bz}(\l)v_{\vec k}\right\|^2\\
&\leq\sup_{\l\in\g}\|R_{\bz}(\l)\|^4\left\|\sum_{{\vec k}\in\bz^d}\langle\d_{\vec\jmath},\Phi(\l)\d_{\vec k}\rangle\ v_{\vec k}\right\|^2.
\end{align*}
The thesis now amounts to show that, for any multi-index ${\bf\a}=(\a_1,\dots,\a_d)$, the sequence $\sup_{\l\in\g}\left\|{\vec\jmath}^{\bf \a}\sum_{{\vec k}\in\bz^d}\langle\d_{\vec\jmath},\Phi(\l)\d_{\vec k}\rangle\ v_{\vec k}\right\|$ is bounded. 

Passing to the Fourier transform on the torus $\bt^d$, and setting $v(\th):=\sum_{{\vec k}\in\bz^d}e^{i{\vec k}}v_{\vec k}$, we have
\begin{align*}
\sup_{\l\in\g}&\left\|{\vec\jmath}^{\bf \a}\sum_{{\vec k}\in\bz^d}\langle\d_{\vec\jmath},\Phi(\l)\d_{\vec k}\rangle
\ v_{\vec k}\right\|=\\
=&\sup_{\l\in\g}\left\|{\vec\jmath}^{\bf \a}\int_{\bt^d}\sum_{{\vec k}\in\bz^d}e^{i{\bf (k-j)\th}}\frac{2s(\l)\sum_{l=1}^{d}\cos2\pi\th_{l}}{s(\l)-2\sum_{l=1}^{d}\cos2\pi\th_{l}}
v_{\vec k}\, d\th\right\|\\
=&\sup_{\l\in\g}\left\|(-i)^{|\a|}\int_{\bt^d}e^{-i{\vec\jmath\th}}\ 
\frac{\partial^{|\a|}}{\partial\th^\a}\left(\frac{2s(\l)
\sum_{l=1}^{d}\cos2\pi\th_{l}}
{s(\l)-2\sum_{l=1}^{d}\cos2\pi\th_{l}}v({\bf\th})\right)\, d\th\right\|\\
\leq&\sup_{\l\in\g}\sup_{\th\in\bt^d}
\left\|\frac{\partial^{|\a|}}{\partial\th^\a}\left(\frac{2s(\l)
\sum_{l=1}^{d}\cos2\pi\th_{l}}
{s(\l)-2\sum_{l=1}^{d}\cos2\pi\th_{l}}v({\bf\th})\right)\right\|.
\end{align*}
Since $v_{\vec k}$ is rapidly decreasing, $v(\th)\in\cc^\infty(\bt^d,\ell^2(\bz))$. The thesis follows.
\end{proof}

We end the present section by pointing out the following fact. For $v\in\cs_0$, define on the Weyl operators $\a_t(W(v)):=W(T_tv)$. We obtain a one--parameter group of $*$--automorphisms 
$t\mapsto\a_t$ on the CCR algebra $\ccr(\cs_0)$. Namely, $(\ccr(\cs_0),\a)$ is the dynamical system which is of interest in our context.

\section{Thermodynamical states for comb graphs in the condensation regime}
\subsection{General results}

We will consider here  the comb $G^{d}=\bz^{d} \comb \bz$  with the finite volume approximations $\La_{n}=X_n\comb Y_n$, where $X_n=(\bz_{2n+1})^{d}$ (periodic boundary condition on the base graph), and $Y_n$ is the line graph $\{-n,\dots,0,\dots,n\}$. Our aim is to study the states $\om_{n}$ relative to the Gibbs grand canonical ensemble on the finite volume approximations $\La_n$ with chemical potential $\mu_{n}$, and the existence of the limit state $\om$ on the comb graph $G^d$ in the condensation regime, that is when $\m_n\to0$. The case $\m_n\to\m<0$ of non condensation regime, presents no further tecnical difficulties, and is described in Theorem \ref{regi}. To avoid technicalities, 
we suppose $\mu_{n}<0$.

 Let $A_{\La_n}$ be the adjacency matrix of the comb graph $\La_n$.  The matrices
 $A_{\La_n}$ can all be considered as operators acting on $\ell^{2}(VG^d)$,
  if we identify $X_n\comb Y_n$ with $[-n,n]^d\comb[-n,n]\subset\bz^d\comb\bz$.  
We have to study the limit behaviour of    
$$
\om_{n}(a^{+}(\xi)a(\eta)),\qquad\xi,\eta\in\cs_{0}.
$$

Let us denote by $H_n=(\|A\|-\m_n)I-A_{\La_n}$  the Hamiltonian on $\La_n$ with chemical potential $\m_n$. We want to compute the limit
\begin{equation}\label{2pflim}
\lim_{n}\langle \eta, (e^{\beta H_n}-I)^{-1}\xi \rangle,
\end{equation}
for suitable vectors $\eta,\xi$.
We first write $$(e^{\beta H_n}-I)^{-1}=\left((e^{\beta H_n}-I)^{-1}-(\beta H_n)^{-1}\right)+(\beta H_n)^{-1}.$$
Then

\begin{Lemma}\label{reductiontoresolvent}
$(e^{\beta H_n}-I)^{-1}-(\b H_n)^{-1}$ converges to  $(e^{\beta H}-I)^{-1}-(\b H)^{-1}$ in the strong operator topology.
\end{Lemma}
\begin{proof}
Indeed the function  $(e^{\beta \l}-I)^{-1}-(\b\l)^{-1}$ is continuous on $[0,\infty)$, hence the result follows by Proposition \ref{StrongConv}.
\end{proof}

We have therefore reduced the computation of (\ref{2pflim}) to the computation of 
\begin{equation*}
\lim_{n}\langle \eta, (\beta H_n)^{-1}\xi \rangle.
\end{equation*}

Set $\l_n:= \norm{A}-\mu_n$. Since $H_n^{-1}=R_{\La_n}(\l_n)$ by definition, the identity of Proposition \ref{Resolvent} implies
\begin{equation}\label{decomposition}
H_n^{-1}=I \otimes R_{Y_{n}}(\l_{n}) + \Phi_{n}\otimes R_{Y_{n}}(\l_{n}) P_{0} R_{Y_{n}}(\l_{n}),
\end{equation}
where $P_{0} = |\d_{0}\rangle \langle \d_{0}|$ (in bra--ket notation), and $\Phi_{n} =\left(I-\langle\d_0,R_{Y_n}(\l_n)\d_0\rangle A_{X_n}\right)^{-1}A_{X_n}$.
Let us notice that, by Lemma \ref{convergence} $(iii)$, $\langle\d_0,R_{Y_n}(\l_n)\d_0\rangle\to(2d)^{-1}$, hence we set
\begin{equation}\label{seteps}
\langle\d_0,R_{Y_n}(\l_n)\d_0\rangle=\frac{1}{2(d+\eps_n)},
\end{equation}
and we have $\eps_n\to0$.
Setting $\eta=\sum_{\vec\jmath}\d_{\vec\jmath}\otimes\eta_{\vec\jmath}$, $\xi=\sum_{\vec k}\d_{\vec k}\otimes\xi_{\vec k}$ we have
\begin{equation}\label{decompH-1}
\langle \eta,  H_n^{-1}\xi \rangle=
\langle\eta,I\otimes R_{Y_{n}}(\l_{n})\xi\rangle
 + \sum_{\vec\jmath, \vec k}\langle\d_{\vec\jmath},\Phi_n \d_{\vec k}\rangle
\langle\eta_{\vec\jmath},R_{Y_{n}}(\l_{n})P_0 R_{Y_{n}}(\l_{n})\xi_{\vec k}\rangle.
\end{equation}

We first observe that, by Proposition \ref{StrongConv},
\begin{equation}\label{nonPertSummand}
\lim_n \langle\eta,I\otimes R_{Y_{n}}(\l_{n})\xi\rangle=
\langle\eta,I\otimes R_{\bz}(2\sqrt{d^2 +1})\xi\rangle
\end{equation}

Let us now compute $\displaystyle\lim_n\langle\d_{\vec\jmath},\Phi_n\d_{\vec k}\rangle$.
Making use of discrete Fourier transform we get
$$
\langle\d_{\vec\jmath},\Phi_n\d_{\vec k}\rangle=
\begin{cases}
2d(d+\eps_n)\displaystyle
\int_{\bt^d}\frac{ (\frac1d\sum_{i=1}^d\cos\th_i)\cos((\vec\jmath-\vec k)\vec\th)}{\eps_n+\sum_{i=1}^d(1-\cos\th_i)}\ dm_n(\vec\th)&\vec\jmath,\vec k \in [-n,n]^d\\
0&\vec\jmath,\vec k \not\in [-n,n]^d
\end{cases}
$$
where $m_n$ is the normalized measure on $\bt^d$ given by
\begin{equation*}
m_n=\frac{1}{(2n+1)^d}\sum_{\vec\jmath\in [-n,n]^d}\d_{\frac{2\pi}{2n+1}\vec\jmath}.
\end{equation*}
We now write
\begin{equation}\label{decompPhi}
\int_{\bt^d}\frac{ (\frac1d\sum_{i=1}^d\cos\th_i)\cos((\vec\jmath-\vec k)\vec\th)}{\eps_n+\sum_{i=1}^d(1-\cos\th_i)}\ dm_n(\vec\th)=
k_n + Q_n(\d_{\vec\jmath},\d_{\vec k}),
\end{equation}
where 
\begin{align}
&k_n=\int_{\bt^d}\frac{ 1}{\eps_n+\sum_{i=1}^d(1-\cos\th_i)}\ dm_n(\vec\th),
\label{defkn}
\\
&Q_n(\d_{\vec\jmath},\d_{\vec k})=
\int_{\bt^d}\frac{ (\frac1d\sum_{i=1}^d\cos\th_i)\cos((\vec\jmath-\vec k)\vec\th)-1}{\eps_n+\sum_{i=1}^d(1-\cos\th_i)}\ dm_n(\vec\th), 
\end{align}

According to Lemma \ref{reductiontoresolvent} and equations (\ref{nonPertSummand}), (\ref{decompH-1}), (\ref{decompPhi}), we have proved the following.

\begin{Lemma}\label{2pfDec}
\begin{align}
\langle\eta,(e^{\b H_n}-I)^{-1}\xi\rangle&=\notag\\
&=
\langle\eta,\left((e^{\b H_n}-I)^{-1}-(\b H_n)^{-1}\right)\xi\rangle
+
\frac1\b\langle\eta,I\otimes R_{Y_n}(\l_n)\xi\rangle
\label{1sum}\\
&+
\frac{2d(d+\eps_n)}{\b}\sum_{\vec\jmath, \vec k\in[-n,n]^d}Q_n(\d_{\vec\jmath},\d_{\vec k})\langle\eta_{\vec\jmath},R_{Y_n}(\l_n)P_0R_{Y_n}(\l_n)\xi_{\vec k}\rangle
\label{2sum}\\
&+
\frac{2d(d+\eps_n)}\b k_n\sum_{\vec\jmath, \vec k\in [-n,n]^d}\langle\eta_{\vec\jmath},R_{Y_n}(\l_n)P_0R_{Y_n}(\l_n)\xi_{\vec k}\rangle,
\label{3sum}
\end{align}
\end{Lemma}

We now show that  $Q_n(\d_{\vec\jmath},\d_{\vec k})$ converges to
 \begin{equation}\label{defQ}
Q(\d_{\vec\jmath},\d_{\vec k}):=\int_{\bt^d}\frac{ (\frac1d\sum_{i=1}^d\cos\th_i)\cos((\vec\jmath-\vec k)\vec\th)-1}{\sum_{i=1}^d(1-\cos\th_i)}\ dm(\vec\th), 
\end{equation}
where $dm$ denotes the normalized Lebesgue measure on $\bt^d$.

\begin{Prop}\label{Qn-Q}
\[
|Q_n(\d_{\vec\jmath},\d_{\vec k})-Q(\d_{\vec\jmath},\d_{\vec k})|\leq\a_n (1+|\vec\jmath-\vec k|^2)
\]
for a suitable infinitesimal sequence $\a_n$.
\end{Prop}

\begin{proof}
Indeed, let $g(\eps,\vec\nu,\vec\th)=\displaystyle\frac{ (\frac1d\sum_{i=1}^d\cos\th_i)\cos(\vec\nu\vec\th)-1}{\eps+\sum_{i=1}^d(1-\cos\th_i)}$, with $\vec\nu:=\vec\jmath-\vec k$, and write
\begin{align*}
|Q_n(\d_{\vec\jmath},\d_{\vec k})-Q(\d_{\vec\jmath},\d_{\vec k})| & = \left|
\int_{\bt^d}\left(g(\eps_n,\vec\nu,\vec\th) \ dm_n(\vec\th)
\right.\right.
-\left.\left. g(0,\vec\nu,\vec\th) \ dm(\vec\th)\right)
\right|\\
&\leq
\left|
\int_{B(0,r)}\left(g(\eps_n,\vec\nu,\vec\th) \ dm_n(\vec\th)-g(0,\vec\nu,\vec\th) \ dm(\vec\th)\right)
\right|\\
&+
\int_{B(0,r)^c}|g(\eps_n,\vec\nu,\vec\th) - g(0,\vec\nu,\vec\th)|\ dm_n(\vec\th)\\
&+
\left|
\int_{B(0,r)^c}g(0,\vec\nu,\vec\th) \left(dm_n(\vec\th) - dm(\vec\th)\right)
\right|
\end{align*}
We now fix $\d>0$, and observe that  we may find $r$ independent of $\eps_n,\vec\nu$ and $\vec\th$ such that the first summand of the r.h.s. above is bounded by $\d(1+|\vec\nu|^2)$.
Moreover, on $\bt^d\setminus B(0,r)$, $|g(\eps_n,\vec\nu,\vec\th) - g(0,\vec\nu,\vec\th)|\leq\displaystyle\frac{4\eps_n}{r^4}<\d$  for sufficiently big $n$. Finally $\displaystyle |\nabla g(0,\vec\nu,\vec\th)|\leq\frac4{r^4}(1+|\vec\nu|)$ on $\bt\setminus B(0,r)$, hence the third summand is bounded by $\d(1+|\vec\nu|)$ for sufficiently big $n$.
The thesis follows.
 \end{proof}

In order to continue the analysis of the limit state $\om$, we have to study $\lim_nk_n$, where $k_n$ was defined in (\ref{defkn}). As we shall see, this requires to study the low-dimensional case and the high-dimensional case separately.

\subsection{The comb graph with low-dimensional base graph}
 \label{crv1}

\begin{Prop} \label{kntoinfty}
If $d\leq2$, then $k_n\to+\infty$.
\end{Prop} 

\begin{proof}
\[
k_n\geq\int_{\bt^d}\frac{ 1}{\eps_n+\frac{|\vec\th|^2}2}\ dm_n(\vec\th)
\geq\int_{\bt^d}\frac{ 1}{\eps_n+\frac{|\vec\th|^2}2}\ dm(\vec\th)
=const\int_0^\pi\frac{ \th^{d-1}}{\eps_n+\frac{\th^2}2}\ d\th
\]
The thesis follows since the last integral diverges for $d=1,2$ when $\eps_n\to0$.
\end{proof}

 Now we may prove the main result of this subsection.

 \begin{Thm} \label{combdiv}
     Let $\xi,\eta\in\cs_{0}$.  For each sequence $\m_n\to0$,
     \begin{equation*}
    \om_{n}(a^{+}(\xi)a(\eta))=k'_{n}\langle \xi,v_{n}\rangle\langle
    v_{n},\eta\rangle +C_{n}(\xi,\eta)\,,
     \end{equation*}
     where $k'_{n}\to+\infty$, $\langle v_{n},\eta\rangle\to\langle v,\eta\rangle$ for any $\eta\in\cs_0$, with $v_n\in V\La_n$ and $v_n \to v$, the generalized Perron--Frobenius vector on $G^d$ described in Lemma \ref{convergence},  and $C_{n}$ converges to a sesquilinear form with domain
     containing $\cs_{0}$.
 \end{Thm}

\begin{proof}

Let us observe that, with $w_n := \norm{R_{Y_n}(\l_n)\d_0}^{-1}R_{Y_n}(\l_n)\d_0$,
\begin{align*}
\langle\eta_{\vec\jmath},R_{Y_n}(\l_n)P_0R_{Y_n}(\l_n)\xi_{\vec k}\rangle
&=
\langle\eta_{\vec\jmath},R_{Y_n}(\l_n)\d_0\rangle\langle R_{Y_n}(\l_n)\d_0,\xi_{\vec k}\rangle\\
&=
\|R_{Y_n}(\l_n)\d_0\|^2\langle\eta_{\vec\jmath},w_n\rangle\langle w_n,\xi_{\vec k}\rangle.
\end{align*}
Denote by $v_n:= u_n\otimes w_n$, where, here and in the following, we use the definitions in Lemma \ref{convergence} for the vectors $u_n,u,w,v$.

On the one hand, we get
\begin{align*}
\sum_{\vec\jmath, \vec k}\langle\eta_{\vec\jmath},R_{Y_n}(\l_n)P_0R_{Y_n}(\l_n)\xi_{\vec k}\rangle
=
\|R_{Y_n}(\l_n)\d_0\|^2\langle\eta,v_n\rangle\langle v_n,\xi\rangle,
\end{align*}
and
\begin{align*}
\langle v-v_n,\xi\rangle
&=
\left|\sum_{\vec\jmath}
\langle u,\d_{\vec\jmath}\rangle\langle w, \xi_{\vec\jmath}\rangle - 
\langle u_n,\d_{\vec\jmath}\rangle\langle w_n, \xi_{\vec\jmath}\rangle
\right| \\
&\leq \sum_{\vec\jmath\not\in[-n,n]^d}\|\xi_{\vec\jmath}\| + 
\|w-w_n\| \sum_{\vec\jmath\in[-n,n]^d}
\|\xi_{\vec\jmath}\|\\
\end{align*}
which tends to 0, since $w_n\to w$ in norm, and the sum $\sum_{\vec\jmath}
\|\xi_{\vec\jmath}\|$ is finite, since $\xi$ is in $\cs_0$. We have thus proved that the term (\ref{3sum}) in Lemma \ref{2pfDec} gives the first summand in the statement, with $k'_n=\b^{-1}2d(d+\eps_n)k_n\|R_{Y_n}(\l_n)\d_0\|^2$.

Again by Lemma \ref{StrongConv}, the term (\ref{1sum}) in Lemma \ref{2pfDec} converges to
$$
\langle\eta,\left((e^{\b H}-I)^{-1}-(\b H)^{-1}\right)\xi\rangle
+ \frac1\b\langle\eta,I\otimes R_{\bz}(\l_\infty)\xi\rangle,
$$
where $\l_{\infty} := 2\sqrt{d^2+1}$. As for the term (\ref{2sum}) in Lemma \ref{2pfDec}, we want to show that it converges to 
\[
\frac{2d^2}\b\sum_{\vec\jmath, \vec k}Q(\d_{\vec\jmath},\d_{\vec k})\langle\eta_{\vec\jmath},R_{\bz}(\l_\infty)P_0R_{\bz}(\l_\infty)\xi_{\vec k}\rangle,
\]
or, equivalently, that
\[
\sum_{\vec\jmath, \vec k\in [-n,n]^d}Q_n(\d_{\vec\jmath},\d_{\vec k})\langle\eta_{\vec\jmath},w_n\rangle\langle w_n,\xi_{\vec k}\rangle
\to
\sum_{\vec\jmath, \vec k}Q(\d_{\vec\jmath},\d_{\vec k})\langle\eta_{\vec\jmath},w\rangle\langle w,\xi_{\vec k}\rangle
\]
where $Q$ was defined in (\ref{defQ}). Indeed
\begin{align*}
&\left|
\sum_{\vec\jmath, \vec k}Q(\d_{\vec\jmath},\d_{\vec k})\langle\eta_{\vec\jmath},w\rangle\langle w,\xi_{\vec k}\rangle
-
\sum_{\vec\jmath, \vec k\in [-n,n]^d}Q_n(\d_{\vec\jmath},\d_{\vec k})\langle\eta_{\vec\jmath},w_n\rangle\langle w_n,\xi_{\vec k}\rangle
\right|
\\
&\leq\sum_{\vec\jmath, \vec k\in[-n,n]^d}|
Q(\d_{\vec\jmath},\d_{\vec k})-Q_n(\d_{\vec\jmath},\d_{\vec k})|\ \|\eta_{\vec\jmath}\|\ \|\xi_{\vec k}\|  \\
&+\sum_{\vec\jmath, \vec k\not\in[-n,n]^d}
|Q(\d_{\vec\jmath},\d_{\vec k})|\ \|\eta_{\vec\jmath}\|\ \|\xi_{\vec k}\| \\
& + \sum_{\vec\jmath, \vec k\in[-n,n]^d}
|Q_n(\d_{\vec\jmath},\d_{\vec k})|\ |\langle\eta_{\vec\jmath},w\rangle\langle w,\xi_{\vec k}\rangle -
\langle\eta_{\vec\jmath},w_n\rangle\langle w_n,\xi_{\vec k}\rangle |.
\end{align*}
According to Proposition \ref{Qn-Q}, the first summand on the r.h.s. is majorized by
\[
\a_n\sum_{\vec\jmath, \vec k}(1+|\vec\jmath-\vec k|)^2\|\eta_{\vec\jmath}\|\ \|\xi_{\vec k}\|
\]
which tends to zero since $\a_n$ does and the sum is finite because $\eta,\xi\in\cs_0$.

Since, by definition of $Q$, $|Q(\d_{\vec\jmath},\d_{\vec k})|\leq (1+|\vec\jmath-\vec k|)^2$, the second summand
is majorized by
\[
\sum_{\vec\jmath, \vec k\not\in[-n,n]^d}(1+|\vec\jmath-\vec k|)^2\|\eta_{\vec\jmath}\|\ \|\xi_{\vec k}\|
\]
which tends to zero since the sum is finite as above.

In the last summand, we again have 
$|Q_n(\d_{\vec\jmath},\d_{\vec k})|\leq (1+|\vec\jmath-\vec k|)^2$, and 
\begin{align*}
|\langle\eta_{\vec\jmath},w\rangle\langle w,\xi_{\vec k}\rangle -
\langle\eta_{\vec\jmath},w_n\rangle\langle w_n,\xi_{\vec k}\rangle |
\leq
2\|w-w_n\|\ \|\eta_{\vec\jmath}\|\ \|\xi_{\vec k}\|.
\end{align*}
Therefore the last summand is bounded by
\[
2\|w-w_n\|\sum_{\vec\jmath, \vec k}(1+|\vec\jmath-\vec k|)^2\|\eta_{\vec\jmath}\|\ \|\xi_{\vec k}\|,
\]
hence tends to 0 as before.
\end{proof}

\begin{Rem}
\itm{i} Theorem \ref{combdiv} tells us that it is impossible to construct any locally normal states describing BEC (i.e. whenever $\m_n\to0$) on the combs $\bz^{d} \comb \bz$, $d=1,2$, see Lemma 
\ref{add2}. In addition, the divergence of the two--point function depends only on the amount of condensate since $C_n$ in Theorem \ref{combdiv} converges to a sesquilinear form which is finite on $\cs_0$.

\itm{ii} Let us observe that the lack of a locally normal thermodynamic state describing condensation does not mean that physically the condensation does not occur. It indeed means that in nonhomogeneous networks, particles condensate even in the configuration space, due to the shape of the wave function of the ground state. Then more and more particles tend to lay in the low energy spectrum and along the base space. The system cannot accommodate them.
\end{Rem}
 
\subsection{The comb graph  with high-dimensional base graph}

We will show here that, for the comb $G^{d}=\bz^{d} \comb \bz$, $d\geq3$, it is possible to construct infinite volume locally normal KMS states. As before, we will consider the finite volume approximations $\La_{n}=X_n\comb Y_n$ of $G^{d}$, where $X_n=(\bz_{2n+1})^{d}$ (periodic boundary condition on the base graph), and $Y_n$ is the line graph $\{-n,\dots,0,\dots,n\}$. Our aim is to show that, for a carefully chosen sequence of chemical potentials $\m_n\to0$, we obtain locally normal thermodynamical states exhibiting Bose-Einstein condensation.

Recalling the discussion above, the estimate from below of $k_n$ in Proposition \ref{kntoinfty} does not imply that $k_n\to\infty$ when $d\geq3$. In order to describe its behavior we have to split $k_n$ in two parts, the zero component of the integral and the rest, $k_n=k^0_n+k^+_n$, with
\begin{align*}
&k^0_n=\frac1{(2n+1)^d\eps_n},
\\
&k^+_n=
\int_{\bt^d\setminus\{\vec 0\}}\frac{ 1}{\eps_n+\sum_{i=1}^d(1-\cos\th_i)}\ dm_n(\vec\th). 
\end{align*}
Our aim is to show that
\[
\lim_n \int_{\bt^d\setminus\{\vec 0\}}\frac{ 1}{\eps_n+\sum_{i=1}^d(1-\cos\th_i)}\ dm_n(\vec\th)
=\int_{\bt^d}\frac{ 1}{\sum_{i=1}^d(1-\cos\th_i)}\ dm(\vec\th).
\]

Let us set $\f(\vec\th)=\sum_{j=1}^d (1-\cos\th_j)$. The integrand is positive, hence, for any $\eps>0$ we have, for $n$ large enough,
\[
\begin{matrix}
\displaystyle
\int_{\bt^d\setminus\{\vec 0\}}\frac{dm_n(\vec\vartheta)}{\eps+\f(\vec\th)}&\leq&
\displaystyle
\int_{\bt^d\setminus\{\vec 0\}}\frac{dm_n(\vec\vartheta)}{\eps_n+\f(\vec\th)}&\leq&
\displaystyle
\int_{\bt^d\setminus\{\vec 0\}}\frac{dm_n(\vec\vartheta)}{\f(\vec\th)}\\
\vphantom{A}& & & & \\
\downarrow n\to\infty& & & & \\
\vphantom{A}& & & & \\
\displaystyle
\int_{\bt^d}\frac{dm(\vec\vartheta)}{\eps+\f(\vec\th)} & & 
\begin{matrix}\eps\to0 \\ \longrightarrow\end{matrix} & &
\displaystyle
\int_{\bt^d}\frac{dm(\vec\vartheta)}{\f(\vec\th)}
\end{matrix}
\]
The result will then follow if 
\[
\int_{\bt^d\setminus\{\vec 0\}}\frac{1}{\f(\vec\th)}\ (dm_n(\vec\vartheta) - dm(\vec\vartheta))\to 0.
\]
Such integral can be rewritten as  
$$
\int_{\bt^d\setminus\{\vec 0\}}\left(\frac{1}{\f(\vec\th)}-\frac{1}{|\vec\th|^2}\right)\ (dm_n(\vec\vartheta) - dm(\vec\vartheta))
+
\int_{\bt^d\setminus\{\vec 0\}}\frac{1}{|\vec\th|^2}\ (dm_n(\vec\vartheta) - dm(\vec\vartheta)).
$$

The first term tends to zero because of the following Lemma.

\begin{Lemma}\label{aecont}
Let $F$ be a bounded function on $[-\pi,\pi]^d$ which is continuous but in zero. Then
\[
\int_{[-\pi,\pi]^d\setminus \{\vec0\}}F(\vec\th)\ dm_n(\vec\vartheta) \to
\int_{[-\pi,\pi]^d}F(\vec\th)\  dm(\vec\vartheta)
\]
\end{Lemma}

\begin{proof}
Let $\|F\|_\infty=C$, and, for a given $\eps>0$, choose $\d$ such that $(2\d)^dC<\eps/2$. Clearly $F$ is uniformly continuous on $\left([-\d,\d]^d\right)^c$, therefore there exists an $n$ such that 
\[
\left|
\int_{\left([-\d,\d]^d\right)^c}F(\vec\th)\ dm_n(\vec\vartheta)
-
\int_{\left([-\d,\d]^d\right)^c}F(\vec\th)\  dm(\vec\vartheta)
\right|<\eps/2.
\]
With such choice
\[
\left|
\int_{[-\pi,\pi]^d\setminus \{\vec0\}}F(\vec\th)\ dm_n(\vec\vartheta) -
\int_{[-\pi,\pi]^d}F(\vec\th)\  dm(\vec\vartheta)
\right|<\eps,
\]
and the Lemma is proved.
\end{proof}

The second term tends to 0 as shown by the following Lemma.
\begin{Lemma}\label{quadraticdiv}
    Let $d\geq 3$.  Then
    $$
    \lim_n
    \int_{\bt^d\setminus\{\vec 0\}}\frac{1}{|\vec\th|^2}\ dm_n(\vec\vartheta)
    =
    \int_{\bt^d}\frac{1}{|\vec\th|^2}\ dm(\vec\vartheta).
    $$
\end{Lemma}

\begin{proof}
Since the domains of our integrals are contained in $[-\pi,\pi]^d$, we may assume
$$
m_n=\frac1{(2n+1)^d}\sum_{\vec\jmath\in\bz^d}\d_{\frac{2\pi}{2n+1}\vec\jmath}.
$$
Denote by $C_{\vec\jmath}$ the square 
$$\prod_{i=1}^{d}[\frac{2\pi}{2n+1}\vec\jmath_i,\frac{2\pi}{2n+1}(\vec\jmath_i+1))$$
and by $\vec1$ the vector with components constantly equal to 1, then, for any $\vec\jmath$ such that $j_i>0$, $i=1,...,d$, we have
\begin{equation}
\label{latticeapprox}
\int_{C_{\vec\jmath+\vec1}}\frac{1}{|\vec\th|^2}\ dm_n(\vec\vartheta)
\leq
\int_{C_{\vec\jmath}}\frac{1}{|\vec\th|^2}\ dm(\vec\vartheta)
\leq
\int_{C_{\vec\jmath}}\frac{1}{|\vec\th|^2}\ dm_n(\vec\vartheta).
\end{equation}
As a consequence we have
\begin{align*}
	\int_{\bt^d} \frac{1}{|\vec\th|^2}\ dm(\vec\vartheta)  = 
	2^d \lim_{n\to\infty}\int_{[\frac{2\pi}{2n+1},\pi)^d} \frac{1}{|\vec\th|^2}\ dm(\vec\vartheta) 
	 \leq \lim_{n\to\infty} \int_{\bt^d\setminus\{\vec 0\}}
	\frac{1}{|\vec\th|^2}\ dm_n(\vec\vartheta).
\end{align*}

We now prove the opposite inequality. We decompose the lattice $$\{\frac{2\pi}{2n+1}\vec\jmath:\vec\jmath\in\bz^d\}\cap \bt^d$$ according to the number of non-zero components of $\vec\jmath$. Therefore, setting
$\bt_n^{d+} = \set{\frac{2\pi}{2n+1}\vec k\in (0,\pi]^d : \vec k\in \bz^d}$, we get
\begin{equation}
\label{zerocomp}
	\int_{\bt^d\setminus\{\vec 0\}}\frac{1}{|\vec\th|^2}\ dm_n(\vec\vartheta)=
	\sum_{j=1}^d \begin{pmatrix}d\\j\end{pmatrix}2^j (2n+1)^{j-d}
	\int_{\bt_n^{j+}}\frac{1}{|\vec\th|^2}\ dm_n(\vec\vartheta).
\end{equation}
If $ j\geq3$, we use inequality (\ref{latticeapprox}) to get
\begin{equation}
\label{estimate}
(2n+1)^{j-d}\int_{\bt_n^{j+}}\frac{1}{|\vec\th|^2}\ dm_n(\vec\vartheta)
\leq
(2n+1)^{j-d}\int_{[0,\pi]^j}\frac{1}{|\vec\th|^2}\ dm(\vec\vartheta)
\end{equation}
and note that the r.h.s. tends to zero if $j<d$.

For $j=1,2$, we decompose $\bt_n^{j+}$ as 
$\{\frac{2\pi}{2n+1}\vec 1\}\cup \set{\frac{2\pi}{2n+1}\vec k\in (0,\pi]^d:\vec k\neq\vec1}$. Hence,
again by inequality (\ref{latticeapprox}), we get
\begin{align*}
(2n+1)^{j-d}\int_{\bt_n^{j+}}\frac{1}{|\vec\th|^2}\ dm_n(\vec\vartheta)
&\leq
\frac1{(2n+1)^{d}}\left(\frac{2n+1}{2\pi}\right)^2\frac{1}{j}\\
& + 2^{-j}(2n+1)^{j-d}\int_{\bt^j\setminus B(0,\frac{2\pi}{2n+1})}\frac{1}{|\vec\th|^2}\ dm(\vec\vartheta).
\end{align*}
Both summands on the r.h.s. are infinitesimal. Therefore equations (\ref{zerocomp}) and (\ref{estimate}) give
\begin{equation*}
	\lim_n\int_{\bt^d\setminus\{\vec 0\}}\frac{1}{|\vec\th|^2}\ dm_n(\vec\vartheta)\leq 
	\int_{\bt^d}\frac{1}{|\vec\th|^2}\ dm(\vec\vartheta)
\end{equation*}
\end{proof}

We have proved the following Lemma.
\begin{Lemma}
\[
\lim_n k^+_n = \int_{\bt^d}\frac{ 1}{\sum_{i=1}^d(1-\cos\th_i)}\ dm(\vec\th).
\]
\end{Lemma}
This result, together with the definition of $\Phi_n$ and Proposition \ref{Qn-Q}, gives
\begin{Prop}
\begin{equation}
\label{decompPhiBis}
(2d(d+\eps_n))^{-1}\langle\d_{\vec\jmath},\Phi_n\d_{\vec k}\rangle=
\begin{cases}
k^0_n + k^+_n + Q_n(\d_{\vec\jmath},\d_{\vec k}) & \vec\jmath,\vec k\in[-n,n]^d,\\
0&\vec\jmath,\vec k\not\in[-n,n]^d.
\end{cases}
\end{equation}
There exists an infinitesimal sequence $\a''_n$ such that
\begin{equation}
\label{ineqalpha''}
\left|\left(k^+_n + Q_n(\d_{\vec\jmath},\d_{\vec k})\right) -
(2d^2)^{-1}\langle\d_{\vec\jmath},\Phi\d_{\vec k}\rangle\right|\leq\a''_n(1+|\vec\jmath-\vec k|^2),
\end{equation}
where $\Phi:=2d\ R_{\bz^d}(2d)A_{\bz^d}$, so that, by Fourier transform,
 \begin{equation*}
\langle\d_{\vec\jmath},\Phi\d_{\vec k}\rangle=2d^2\int_{\bt^d}\frac{ (\frac1d\sum_{i=1}^d\cos\th_i)\cos((\vec\jmath-\vec k)\vec\th)}{\sum_{i=1}^d(1-\cos\th_i)}\ dm(\vec\th).
\end{equation*}

\end{Prop}

\subsection{The choice of $\m_n$.}\label{subsec2}

In order to have a finite limit for $k^0_n$ we have to assume that $\lim_n((2n+1)^d\eps_n)^{-1}$ is finite. The following holds.

\begin{Lemma}\label{finitek0}
$$
\lim_n \frac{1}{-\m_n \vol(X_n)}=c<\infty 
\Leftrightarrow
\lim_n k^0_n=\lim_n \frac{1}{\eps_n (2n+1)^d}=c\frac{2d}{\sqrt{d^2+1}}<\infty.
$$
\end{Lemma}

\begin{proof}
By (\ref{seteps}) and Lemma \ref{convergence} $(iii)$, we obtain $\eps_n=\displaystyle\frac{\sqrt{\l_n^2-4}}{2\tanh (n+1)\t_n}-d$, with $2\cosh\t_n=\l_n=\|A_{G^d}\|-\m_n$, and $\|A_{G^d}\|=\l_\infty=2\sqrt{d^2+1}$. Then
\begin{align*}
\frac{1}{\eps_n(2n+1)^d}=\left(\frac{\sqrt{\l_n^2-4}-2d}{2\tanh (n+1)\t_n}(2n+1)^d
+
d\frac{1-\tanh (n+1)\t_n}{\tanh (n+1)\t_n}(2n+1)^d
\right)^{-1}.
\end{align*}
Since $\t_n\to\t_\infty$ with $\cosh\t_\infty=\sqrt{d^2+1}$, $\tanh (n+1)\t_n$ tends to 1 exponentially fast, and the second summand above tends to 0. As for the first summand,
\begin{align*}
\left(\frac{\sqrt{\l_n^2-4}-2d}{2\tanh (n+1)\t_n}(2n+1)^d\right)^{-1}=
\left(\frac{(-2\m_n\l_\infty+\m_n^2)(2n+1)^d}{2(\tanh (n+1)\t_n)(\sqrt{\l_n^2-4}+2d)}\right)^{-1}
\end{align*}
and the latter has a finite limit if and only if $(-\m_n (2n+1)^d)^{-1}$ has a finite limit. The thesis follows since $\vol(X_n)=(2n+1)^d$.
\end{proof}

\begin{Lemma}\label{PreHn-1}
Let us assume $\lim_n\left(-\m_n\vol(X_n)\right)^{-1}=c<\infty$. Then, for any $\eta,\xi\in\cs_0$,  
\begin{align*}
\lim_{n}
\langle \eta, ( H_n)^{-1}\xi \rangle
=\langle\eta,I\otimes R_{\bz}(\l_\infty)\xi\rangle
+\sum_{\vec\jmath,\vec k}
\langle\d_{\vec\jmath},\Phi\d_{\vec k}\rangle\langle\eta_{\vec\jmath},w\rangle
\langle w, \xi_{\vec k}\rangle
 +c\langle\eta,v\rangle\langle v,\xi\rangle.
\end{align*}
\end{Lemma}

\begin{proof}
Let us recall that, by Lemma \ref{2pfDec},
\begin{align*}
\langle\eta,H_n^{-1}\xi\rangle&=
\langle\eta,I\otimes R_{Y_n}(\l_n)\xi\rangle\\
&+
2d(d+\eps)\sum_{\vec\jmath, \vec k\in[-n,n]^d}(Q_n(\d_{\vec\jmath},\d_{\vec k})+k_n^+)\langle\eta_{\vec\jmath},R_{Y_n}(\l_n)P_0R_{Y_n}(\l_n)\xi_{\vec k}\rangle
\\
&+
2d(d+\eps) k_n^0\sum_{\vec\jmath, \vec k\in[-n,n]^d}\langle\eta_{\vec\jmath},R_{Y_n}(\l_n)P_0R_{Y_n}(\l_n)\xi_{\vec k}\rangle.
\end{align*}
The first summand tends to$ \langle\eta,I\otimes R_{\bz}(\l_\infty)\xi\rangle$ by Lemma \ref{StrongConv}. The second summand tends to $\sum_{\vec\jmath,\vec k}
\langle\d_{\vec\jmath},\Phi\d_{\vec k}\rangle\langle\eta_{\vec\jmath},w\rangle
\langle w, \xi_{\vec k}\rangle$ as in the proof of Theorem \ref{combdiv}, with the aid of (\ref{ineqalpha''}).

By Lemma \ref{finitek0} and by  the proof of Theorem \ref{combdiv}, the last summand tends to 
\[ c \frac{4d^3}{\sqrt{d^2+1}}\|R_{\bz}(\l_\infty)\d_0\|^2\langle\eta,v\rangle\langle v,\xi\rangle.\]
The thesis follows by Lemma \ref{convergence} $(iv)$.
\end{proof}

We also have the following.
\begin{Lemma}\label{ResolventUpToNorm}
When $G^d$ is considered as an infinite perturbation of infinitely many disjoint fibres $\bz$, equation \eqref{eq:Resolvent} is valid for any $\l> \|A\|$.
\end{Lemma}
\begin{proof}
Indeed the identity becomes
$$
R_{A}(\l) = I \otimes R_{\bz}(\l) + (I-\langle\d_0,R_\bz(\l)\d_0\rangle A_{\bz^d})^{-1}
A_{\bz^d}\otimes R_{\bz}(\l) P_{0} R_{\bz}(\l),
$$
where $P_{0} = |\d_{0}\rangle \langle \d_{0}|$. Such identity has been proved to hold whenever the operator$(I-\langle\d_0,R_\bz(\l)\d_0\rangle A_{\bz^d})$ is invertible.

Since $\langle\d_0,R_\bz(\l)\d_0\rangle=(\l^{2}-4)^{-1/2}$, we have $(I-\langle\d_0,R_\bz(\l)\d_0\rangle A_{\bz^d})=(\l^{2}-4)^{-1/2}(\sqrt{\l^{2}-4} - A_{\bz^d})$, which is invertible whenever $\sqrt{\l^{2}-4}>\|A_{\bz^d}\|=2d$, namely whenever $\l>2\sqrt{d^2+1}=\|A\|$.
\end{proof}

>From the last Lemma, we have
\begin{align*}
\langle\xi,R_{A}(\l)\eta\rangle 
&=
\sum_{\vec\jmath}\langle\eta_{\vec\jmath},R_{\bz}(\l)\xi_{\vec\jmath}\rangle\\
&+\sum_{\vec\jmath,\vec k}
\langle\d_{\vec\jmath},
(I-\langle\d_0,R_\bz(\l)\d_0\rangle A_{\bz^d})^{-1}A_{\bz^d}
\d_{\vec k}\rangle
\langle\eta_{\vec\jmath}, R_{\bz}(\l) P_{0} R_{\bz}(\l) \xi_{\vec k}\rangle.
\end{align*}
Taking the limit for $\l\to\|A\|^+$ in the equation above, and using Lemma \ref{PreHn-1}, we
conclude
\begin{equation*}
\lim_{n}\langle \eta, H_n^{-1}\xi \rangle
 =\langle \eta, H^{-1}\xi \rangle
 +c\langle\eta,v\rangle\langle v,\xi\rangle.
\end{equation*}

\subsection{Conclusion} 
 \begin{Thm}\label{mainthm}
Let  $d\geq3$,  $G^d $ denote the comb $ \bz^{d} \comb \bz$, $\La_n=X_n\comb Y_n$ be its approximation. Moreover, $H_n=(\|A\|-\m_n)I-A_{\La_n}$ denotes the Hamiltonian on $\La_n$ with chemical potential $\m_n$, $H=\|A\|I-A$ denotes the pure hopping Hamiltonian on $G^d$, and $v$ denotes the Perron-Frobenius generalized vector for $A$ considered in Lemma \ref{convergence}. Then:
\item{$(i)$} the adjacency operator $A$ for $G^d$ is transient,
\item{$(ii)$}  $\cs_{0}\subset\cd\big( R_{G^d}(\|A\|)^{1/2}\big)$.

Assume now 
\begin{equation}
\label{assumption}
\lim_n \frac1{-\m_{n}\vol(X_n)}=c<+\infty.
\end{equation}
 Then
\begin{align}
\lim_{n}&\langle \eta, (e^{\beta H_n}-I)^{-1}\xi \rangle=
\langle \eta, (e^{\beta H}-I)^{-1} \xi \rangle + \frac{c}{\b}\langle\eta,v\rangle\langle v,\xi\rangle
\quad \forall \xi,\eta\in\cs_{0},\label{2pfunct}
\\
\lim_{n}&\,\t_{\La_{n}} \bigl( (e^{\beta H_n}-I)^{-1} \bigr)=\,
\t\bigl( (e^{\beta H}-I)^{-1} \bigr)\,,\label{density}
\end{align}
where  $\t_{\La_{n}}$ is the normalized trace on $\La_n$.
\end{Thm}
\begin{proof}
Statement $(ii)$ means that, for any $\xi\in\cs_0$, 
$\lim_{\l\to\|A\|^+}\langle\xi,R_{G^d}(\l)\xi\rangle$ is finite, and this follows from Lemmas \ref{PreHn-1}, and 
\ref{ResolventUpToNorm}.
Clearly $(ii)\Rightarrow(i)$ when $\xi=\d_J$.
\\
Equation (\ref{2pfunct}) summarizes the results proved above in this section.
\\
We now prove equation (\ref{density}). First we use Proposition \ref{traceconv} to show that, as in Lemma \ref{reductiontoresolvent}, we only have to compute $\lim_{n}\t_{\La_{n}}(H_n^{-1})$. Using formula \eqref{decomposition}, we have
\[
\t_{\La_{n}}(H_n^{-1})=
\t_{[-n,n]} (R_{Y_{n}}(\l_{n})) + \t_{[-n,n]^d}(\Phi_{n}) \t_{[-n,n]}(R_{Y_{n}}(\l_{n}) P_{0} R_{Y_{n}}(\l_{n}))
\]
where $\t_{[-n,n]^j}$ denotes the normalized trace on $[-n,n]^j$. Again by Proposition  \ref{traceconv} we get 
$\t_{[-n,n]} (R_{Y_{n}}(\l_{n}))\to\t_\bz (R_{\bz}(\|A\|))$ and 
$\t_{[-n,n]}(R_{Y_{n}}(\l_{n}) P_{0} R_{Y_{n}}(\l_{n}))\to
\t_\bz(R_{\bz}(\|A\|) P_{0} R_{\bz}(\|A\|))=0$, since $R_{\bz}(\|A\|) P_{0} R_{\bz}(\|A\|)$ is finite rank. The result follows if we show that $\t_{[-n,n]^d}(\Phi_{n})$ is bounded. 
Let us recall that, by \eqref{decompPhiBis},
$\langle\d_{\vec\jmath},\Phi_n\d_{\vec k}\rangle =2d(d+\eps_n)\bigl(k_n^0+k_n^+ +Q_n(\d_{\vec\jmath},\d_{\vec k})\bigr)$. By formula \eqref{ineqalpha''},
\[
|\t_{[-n,n]^d}(k_n^+ +Q_n)|\leq\a''_n+(2d^2)^{-1}\t_{[-n,n]^d}(\Phi),
\]
while $k_n^0$ is bounded by (\ref{assumption}).
\end{proof}

 \begin{Rem}
 \itm{i} According to the Theorem above, in order to get a finite contribution for the condensate in the two-point function,  condition $|\m_n|\geq const\ n^{-d}$ should be satisfied, for a suitable positive constant (cf. Lemma \ref{finitek0}). In this case, again by the previous theorem, the condensate does not contribute to the density. This is because the condensate is spatially distributed according to the Perron-Frobenius vector, namely around the base graph, therefore the condensate in $\La_n$ grows as $n^{d}$, while the volume grows as $n^{d+1}$.
 
 \itm{ii} Conversely, if we try to construct the thermodynamical state as a limit with fixed density, in particular 
choosing the inverse temperature $\b>0$, a parameter $k>0$, and  $\m_n$ in such a way that
 \begin{equation} 
 \label{hhh}
     \r_{\La_{n}}(\b,\m_{n})=\r_{c}(\b)+k,
 \end{equation}
we do not get a finite two-point function on local vectors.

Indeed, according to the proof of equation (\ref{density}), the only term there which depends on the sequence $\m_n$ is $\t_{[-n,n]^d}(\Phi_{n}P_{u_n})$, or, equivalently, $\t_{\La_{n}}(H_n^{-1}P_{v_n})$. We have
\[
\t_{\La_{n}}(H_n^{-1}P_{v_n})=(\|A\|-\m_n-\|A_{\La_n}\|)^{-1}(2n+1)^{-d-1},
\]
 which, together with equation (\ref{hhh}), gives $\lim_n|\m_n|^{-1}(2n+1)^{-d-1}=k$. But with this choice $\langle\eta, H_n^{-1}P_{v_n}\xi\rangle$  behaves like $|\m_n|^{-1}(2n+1)^{-d}\langle \eta,v_n\rangle\langle v_n,\xi\rangle$, which diverges as soon as $\langle \eta,v_n\rangle\langle v_n,\xi\rangle\ne0$.
 
\itm{iii}
A non locally normal infinite--volume KMS state
can be always constructed on the space 
$\cd_{0}:=\bigcup_{\eps>0}P_{\eps}\ell^{2}(G_{d})$, $P_{\eps}$ being 
the spectral projection of the Hamiltonian $H=\|A\|-A$ corresponding 
the the spectral subspace $[\eps,+\infty)$. 
The two--point funcion is given for $X,Y\in\cd_{0}$, 
$$
\om(a^{+}(\xi)a(\eta))=\big\langle
\xi,\big(e^{\b(\|A\|-A)}-1\big)^{-1}\eta\big\rangle\,.
$$

This can be obtained as infinite volume limit of any sequence of 
finite volume Gibbs states based on any sequence of chemical 
potentials $\m_{n}\to0$. As the elements 
of $\cd_{0}$ are formally orthogonal to the Perron Frobenius eigenvector, no 
amount of condensate can be appreciated in such non locally normal 
state. 
 \end{Rem}

We end the present section by showing that the locally normal states described in the previous theorem are KMS for the dynamics generated on $\ccr(\cs_0)$ by the one--parameter group of Bogoliubov transformations $e^{itH}$.
\begin{Thm}
\label{addd1}
If $d\geq3$ then the states $\om^{(c)}$, with two-point function
\[
\om^{(c)}(a^{+}(\xi)a(\eta))= 
\langle \eta, (e^{\beta H}-I)^{-1} \xi \rangle + \frac{c}{\b }\langle\eta,v\rangle\langle v,\xi\rangle
,
\]
on the CCR algebra $\ccr(\cs_0)$ associated to $\cs_{0}$ are $\b$--KMS w.r.t. the time evolution 
$\a_{t}$ induced on $\ga$ by $T_{t}:=e^{itH}$.
\end{Thm}
\begin{proof}
By Proposition \ref{add4}, $T_{t}\cs_{0}=\cs_{0}$, thus it induces a
one--parameter group of automorphisms of $\ga$, by putting
$$
\a_{t}(W(\xi)):=W(T_{t}\xi)\,.
$$

By taking into account the form of the two--point function for the
$\om^{(c)}$ (cf. \cite{BR2}, pag. 79), they are automatically KMS, provided that the functions
$$
t\in\br\mapsto\om^{(c)}(W(\xi)\a_{t}(W(\eta)))\,,\quad \xi,\eta\in\cs_{0}
$$
are continuous. Notice that (cf. \cite{BR2}, pag. 42)
\begin{equation}
\label{addd2}
\om^{(c)}(W(\xi))=\exp \bigg\{ -\frac{\|\xi\|^{2}}{4} \bigg\}
\exp \bigg\{ -\frac{\om^{(c)}(a^{+}(\xi)a(\xi))}{2} \bigg\} \,,
\end{equation}
and, by using the commutation rule,
\begin{equation}
\label{addd3}
\om^{(c)}(W(\xi)\a_{t}(W(\eta)))=e^{-i\im \langle \xi, T_{t}\eta\rangle}
\om^{(c)}(W(\xi+T_{t}\eta))\,.
\end{equation}

Thus, by taking into account \eqref{addd2}, \eqref{addd3}, it is
enough to show that
$$
t\in\br\mapsto\om^{(c)}(a^{+}(\xi+T_{t}\eta)a(\xi+T_{t}\eta))
$$
is continuous whenever $\xi,\eta\in\cs_{0}$. Indeed,
\begin{align*}
\om^{(c)}&(a^{+}(\xi+T_{t}\eta)a(\xi+T_{t}\eta))=\\
=&\langle  \xi+\eta, v\rangle\langle v, \xi+\eta\rangle
+\langle \xi+T_{t}\eta,(e^{\b H}-1)^{-1}(\xi+T_{t}\eta) \rangle\\
=&|\langle  v,  \xi+\eta \rangle|^2
+\langle \xi, (e^{\b H}-1)^{-1}\xi \rangle
+\langle T_t (e^{\b H}-1)^{-1/2}\eta, (e^{\b H}-1)^{-1/2}\xi \rangle+\\
+&\langle (e^{\b H}-1)^{-1/2}\xi, T_t(e^{\b H}-1)^{-1/2}\eta\rangle
+\langle \eta, (e^{\b H}-1)^{-1}\eta \rangle
\end{align*}
The proof follows as $\xi\in\cs_{0}$ implies that $\xi\in\cd((e^{\b 
H}-1)^{-1/2})$, as shown in Theorem \ref{mainthm} $(ii)$.
\end{proof}
 
 \medskip\par\noindent{\it Acknowledgements.}
      The first--named author would like to thank L. Accardi for the invitation to the 7th Volterra--CIRM
      International School: "Quantum Probability and Spectral Analysis on Large Graphs", where discussions with
      several participants were very inspiring for the present investigation. He is also grateful to M. Picardello
      for useful discussions on the topic. The second and third named authors would like to thank P. Kuchment for his invitation to the workshop "Analysis on Graphs
and Fractals" at the University of Wales, Cardiff, where some of the results contained in the present paper were
presented.


\begin{thebibliography}{9999}
    
\bibitem{ABO} Accardi L., Ben Ghorbal A., Obata N. 
{\it Monotone independence, Comb graphs and Bose--Einstein condensation},
Infin. Dimens. Anal. Quantum Probab. 
Relat. Top. {\bf 7} (2004), 419--435.

\bibitem{BR1} Bratteli O., Robinson D. W.
{\it Operator algebras and quantum statistical mechanics I},
Springer, Berlin--Heidelberg--New york, 1979.

\bibitem{BR2} Bratteli O., Robinson D. W.
{\it Operator algebras and quantum statistical mechanics II},
Springer, Berlin--Heidelberg--New york, 1981.

\bibitem{BCRSV} Burioni R., Cassi D., Rasetti M., Sodano P., Vezzani A.
{\it Bose--Einstein condensation on inhomogeneous complex networks},
J. Phys. B {\bf 34} (2001), 4697--4710.

\bibitem{Fidaleo} Fidaleo, F. {\it Harmonic analysis on perturbed Cayley Trees},  J. Funct. Anal., to appear (arXiv:1003.0083 [math.FA]).

\bibitem{FGI} Fidaleo F., Guido D., Isola T., work in progress.

\bibitem{GILa03}  Guido D., Isola T., Lapidus M. L.  {\it Ihara's
 zeta function for periodic graphs and its approximation in the
 amenable case}, J. Funct. Anal. {\bf 255} (2008), 1339--1361. 

\bibitem{HO} Hora A., Obata N. {\it Quantum probability and spectral analysis on graphs},
Springer--Verlag, Berlin, 2007.

\bibitem{Luck} L\"uck  W. {\it $L\sp 2$-invariants: theory and applications to geometry and $K$-theory}, Springer--Verlag, Berlin, 2002.

\bibitem{M} Matsui T. {\it BEC of free Bosons on networks},  Infin. Dimens. Anal. Quantum Probab. Relat. Top.  {\bf 9}  (2006),  1--26.

\bibitem{PF} Pastur L., Figotin A. {\it Spectra of random and almost--periodic
operators}, Springer--Verlag, Berlin, 1992.

\bibitem{Ped} Pedersen G. K. {\it C$^{*}$-algebras and their 
automorphism groups},  Academic Press, London, 1979.

\bibitem{RS} Reed M., Simon B. {\it Analysis of operators}, 
Academic Press, New York--London 1978.

\bibitem{Seneta} Seneta E. {\it Nonnegative matrices and Markov chains}, Springer--Verlag, New York, 1981.

\end{thebibliography}
\end{document}